\documentclass{article}

\usepackage{arxiv}

\usepackage[utf8]{inputenc} % allow utf-8 input
\usepackage[T1]{fontenc}    % use 8-bit T1 fonts
\usepackage{hyperref}       % hyperlinks
\usepackage{url}            % simple URL typesetting
\usepackage{booktabs}       % professional-quality tables
\usepackage{amsfonts}       % blackboard math symbols
\usepackage{nicefrac}       % compact symbols for 1/2, etc.
\usepackage{microtype}      % microtypography
\usepackage{graphicx}
\usepackage{natbib}
\usepackage{doi}
\usepackage{amsmath,amssymb}
\usepackage{enumitem}

\numberwithin{equation}{section}
\newtheorem{theorem}{Theorem}[section]
\newtheorem{lemma}[theorem]{Lemma}
\newtheorem{remark}[theorem]{Remark}
\newtheorem{proposition}[theorem]{Proposition}
\newtheorem{definition}[theorem]{Definition}

\newtheorem{proof}{Proof}
\DeclareMathOperator{\sign}{sign}
\DeclareMathOperator{\spn}{span}
\DeclareMathOperator{\dimm}{dim}
\DeclareMathOperator{\codim}{codim}
\DeclareMathOperator*{\esssup}{ess\,sup}
\DeclareMathOperator\erf{erf}

\newcommand{\as}[1]{\left\vert#1\right\vert}
\newcommand{\magg}[1]{\left\vert#1\right\vert ^{2}}
\newcommand{\norm}[1]{\left\Vert#1\right\Vert}
\newcommand{\ds}{\displaystyle}

\title{Biological aggregations from spatial memory and nonlocal advection}

\author{{Di Liu}\thanks{Co-first authors.} \\
	School of Mathematics\\
	Harbin Institute of Technology\\
	Harbin, China \\
	\And
        {Yurij Salmaniw$^{*,}$}\thanks{Corresponding author.} \\
	Department of Mathematical and Statistical Sciences\\
	University of Alberta\\
	Edmonton, Canada \\
	\texttt{salmaniw@ualberta.ca} \\
	\AND
	Jonathan R. Potts \\
        School of Mathematics and Statistics \\
	University of Sheffield \\
	Sheffield, United Kingdom \\
	\And
	Junping Shi \\
	Department of Mathematics \\
	College of William $\&$ Mary \\
        Williamsburg, VA. \\
	\And
	Hao Wang \\
        Department of Mathematical and Statistical Sciences \\
	University of Alberta \\
	Edmonton, Alberta \\
}

\renewcommand{\shorttitle}{D. Liu, Y. Salmaniw et al.}

\hypersetup{
pdftitle={A template for the arxiv style},
pdfsubject={q-bio.NC, q-bio.QM},
pdfauthor={David S.~Hippocampus, Elias D.~Striatum},
pdfkeywords={First keyword, Second keyword, More},
}

\begin{document}
\maketitle

\begin{abstract}
        We investigate a nonlocal single-species reaction-diffusion-advection model that integrates the spatial memory of previously visited locations and nonlocal detection in space, resulting in a coupled PDE-ODE system reflective of several existing models found in spatial ecology. We prove the existence and uniqueness of a H{\"o}lder continuous weak solution in one spatial dimension under some general conditions, allowing for discontinuous kernels such as the top-hat detection kernel. A robust spectral and bifurcation analysis is also performed, providing the rigorous analytical study not yet found in the existing literature. In particular, the essential spectrum is shown to be entirely negative, and we classify the nature of the bifurcation near the critical values obtained via a linear stability analysis. A pseudo-spectral method is used to solve and plot the steady states near and far away from these critical values, complementing the analytical insights. 
\end{abstract}

% keywords can be removed

%%%%%%%%%%%%%%%%%%%%%%%%%%%%%%%%%%%%%%%%%%%%%%%%%%%%%%%%%%%%%%%%%%%%%%%%%%%%%%%%%%%%%%%%%%%%%%%%%%%%%%%%%%%%%%%%%%%%%%%%%%%%%%%%%%%%%%%%%%%%%%%%%%%%%%%%%%%%%%%%%%%%%%%%%%%%%%%%%%%%%%%%%%%%%%%%%%%%%%%%%%%%%%%%%%%%%

\section{Introduction}

\subsection{Background and model formulation}

Spatial memory is a key feature driving the movement of mobile organisms \cite{faganetal2013}.  As organisms move, they gather information about where they have been, building a map that informs future movement decisions.  This process generates a feedback mechanism, whereby previous visitations of favourable locations can cause repeated visits, resulting in the organism confining itself to certain specific areas of the landscape.  In animal ecology, such memory processes have been hypothesised to be foundational in the construction of home ranges \citep{briscoeetal2002, borgeretal2008, vanmoorteretal2009}, small areas where an animal decides to perform its daily activities instead of roaming more widely.  Conversely, memory of unfavourable locations can cause animals to relocate.  For example, memory has been shown to play a key role in migratory movements \citep{bracismuller2017, abrahmsetal2019} and avoiding conspecifics to form territories or home ranges \citep{pottslewis2016, ellisonetal2020}.  

Understanding how memory processes help to shape the space use of animals is thus becoming a question of increasing interest in both empirical ecology \citep{faganetal2013, merkleetal2014, merkleetal2019, lewisetal2021} and mathematical modelling \citep{pottslewis2016, ShiShiWang2021JMB, wangsalmaniw2022}.  From a modelling perspective, a key tool for modelling movement in response to remembered space use is via an advection term in a partial differential equation (PDE).  This advection term is typically nonlocal in space, for both biological and mathematical reasons.  From a biological perspective, nonlocality is important because organisms will generally sense their surrounding environment -- for example through sight, smell, or touch -- and make movement decisions accordingly \citep{benhamou2014, martinezgarciaetal2020}. Moreover, this nonlocality occurs not only in animals but also in cells \citep{armstrongetal2009, painter2023biological}.  Mathematically, nonlocal advection is often crucial for well-posedness and avoiding blow-up of PDEs \citep{bertozzilaurent2007, giunta2022detecting}.

Alongside advection, mathematical models of organism movement typically have a diffusive term, accounting for the aspects of movement that we are not explicitly modelling (such as foraging), and may also have a reaction term representing the births and deaths of organisms.  This leads to the formalism of reaction-diffusion-advection equations (RDAs).  In a one dimensional spatial domain $\Omega$, such an RDA might have the following general form
\begin{align}\label{prototype}
u_t = d u_{xx} + \alpha (u a_x)_x + f(u), \quad x \in \Omega, \ t> 0.
\end{align}
Here, $d>0$ denotes the rate of diffusion (exploratory movement), $\alpha \in \mathbb{R}$ denotes the rate of advection towards ($\alpha <0$) or away from ($\alpha>0$) the environmental covariates described by the function $a(x,t)$, and $f(u)$ describes population changes through birth/death processes. When paired with an appropriate initial/boundary condition, we seek to analyze dynamical behaviours of the solution $u(x,t)$ as it depends on parameters appearing in the equation. 

The aspect of memory then appears in the advection term $a(x,t)$. A recent review paper by \citep{wangsalmaniw2022-2} covers in detail the development of equations to model memory, as well as the related concept of learning, along with a large collection of open problems and directions in this area.  The central idea is to model spatial memory as a map, $k(x,t)$, which evolves over time as the organism learns about its environment \citep{faganetal2013, pottslewis2016, pottslewis2019}.  This map may represent something in the mind of a specific animal, sometimes called a {`cognitive map'} \citep{Harten2020, Peer2021}, or it could represent memory of past animal locations embedded in the environment, e.g. due to animals depositing scent marks or forging trails.

Here, we seek to explore the influence of such a map on the space-use patterns of a single population, $u$. To this end, we describe the evolution of $k(x,t)$ through the ordinary differential equation for each $x \in \Omega$, following \cite{pottslewis2016}
\begin{equation}\label{2a}
k_t = g(u) - ( \mu + \beta u) k, \quad t>0.
\end{equation}
%In this form, the map most accurately describes the memory of direct interactions between members of the population $u$. 
Here, the function $g(\cdot)$ describes the uptake rate of the map $k$ as it depends on the population $u(x,t)$; $\mu\geq0$ describes the rate at which memories fade over time; and $\beta\geq0$ describes a rate at which organisms remove a location from their memory map on revisitation (e.g. if animals want to avoid overuse of a location \citep{pottsetal2022}). 

Note that, for simplicity, we have assumed that all organisms in a population share a common memory map.  This makes it perhaps more amenable to modelling the distribution of cues left on the environment, e.g. scent marks or visual cues \citep{lewismurray1993, moorcroftlewis2006}, rather than memory contained in the minds of animals.  Alternatively, if the population modelled by $u(x,t)$ has some process of relatively-rapid information sharing, then we can view $k(x,t)$ as a shared memory amongst the population (e.g. for social insects this may be valid).   As another example, if $u(x,t)$ is the probability distribution of a single animal (in which case $f(u)=0$ per force) then $k(x,t)$ can be used to model a map in the mind of an individual \citep{pottslewis2016}.

Prototypical examples of the function $g(u)$ might be $g(u)=\rho u$, denoting memory accruing in proportion to animal visitations, or $g(u)=\rho u^2$, denoting uptake of memory when members of $u$ encounter one another (here, $\rho$ is a constant).   However, these can, in principle, both lead to unbounded memory.  Therefore we can either take another functional form, such as $g(u)=\rho u^2/(1+cu)$ (cf. the Holling type II functional response \citep{holling1965}), or modify Equation (\ref{2a}) to the following as in \cite{pottslewis2016}:
\begin{equation}\label{2a-2}
k_t = g(u) (\kappa-k) - ( \mu + \beta u) k,\quad t>0,
\end{equation}
where $\kappa>0$ denotes a theoretical maximal memory capacity.

To combine this mechanism of spatial memory with nonlocal perception, we model nonlocal effects through a spatial convolution:
\begin{equation}\label{1.3}
\overline{k} (x,t) =(G\ast k)(x,t)=\frac{1}{\as{\Omega}} \int_{\Omega} G(x-y) k(y,t) dy.
\end{equation}
Here, the function $G(\cdot)$ is referred to as a \textit{perceptual kernel} or \textit{detection function}, which describes how an animals' ability to perceive landscape information varies with distance \cite{Fagan2017perceptual, wangsalmaniw2022}. Common forms of the detection function $G(\cdot)$ include the \textit{Gaussian} detection function, the \textit{exponential} detection function, or the \textit{top-hat} detection function, each taking the respective forms in $\Omega = \mathbb{R}$:
\begin{align}
G(x) &:= \frac{1}{\sqrt{2 \pi} R} e^{- x^2 / 2 R^2} , \label{detectionkernelG} \\
G(x) &:= \frac{1}{2R} e^{- \as{x}/ R} ,\label{detectionkernelE} \\
G(x) &: =
\begin{cases} 
\frac{1}{2R} , \hspace{0.5cm}- R \leq x \leq R, \cr
0, \hspace{1.0cm} \text{otherwise}. \label{detectionkernelT}
\end{cases} 
\end{align}
Here, $R\geq 0$ is the \textit{perceptual radius}, describing the maximum distance at which landscape features can be distinguished \cite{Fagan2017perceptual}. Roughly, the Gaussian detection function provides the most information far away from the location of observation, whereas the top-hat detection function gives a strict limit on how far the organism can detect information. In general, it is reasonable to assume the detection function satisfies
\begin{enumerate}[label={\roman*}.)]
\item $G (x)$ is symmetric about the origin;
\item $\int_{\mathbb{R}} G (x) dx = 1$;
\item $\lim_{R \to 0^+} G (x) = \delta (x)$;
\item $G(x)$ is non-increasing from the origin.
\end{enumerate}\label{generaldetectionkernel}
Here, $\delta (x)$ denotes the Dirac-delta distribution. Each of the Gaussian, exponential, and top-hat kernels satisfy these properties over $\mathbb{R}$; appropriate modification is sometimes required in a bounded domain. Readers are encouraged to review \cite{Fagan2017perceptual, wangsalmaniw2022} for further discussion on detection kernels and some of the challenges in defining nonlocal kernels near a boundary region.

Taking the advective potential $a(x,t) = \overline{k}(x,t)$, where $k$ solves either \eqref{2a} or \eqref{2a-2}, we combine the equation describing movement \eqref{prototype} with a dynamic spatial map in $\Omega = (-L,L)$, $L>0$, to arrive at the following two systems of equations subject to periodic boundary conditions:
\begin{equation}\label{1a}
\begin{cases}
u_t = d u_{xx} + \alpha ( u \overline{k}_{x} )_{x} + f(u), & x\in (-L,L), \ t>0, \\
k_t = g(u) - ( \mu + \beta u) k, & x \in (-L,L), \ t> 0,
\end{cases}\tag{1.8.a}
\end{equation}
and
\begin{equation}\label{1b}
\begin{cases}
u_t = d u_{xx} + \alpha ( u \overline{k}_{x} )_{x} + f(u), & x\in (-L,L), \ t>0, \\
k_t = g(u)(\kappa - k) - ( \mu + \beta u) k, & x \in (-L,L), \ t> 0.
\end{cases}\tag{1.8.b}
\end{equation}
In either case, we denote by $u(x,0) = u_0 (x)$, $k(x,0) = k_0 (x)$ the initial data, chosen to be $2L$-periodic in $\Omega$. As discussed in \cite{wangsalmaniw2022}, boundary conditions in a nonlocal setting in a bounded domain are highly non-trivial in general. More precisely, it is not clear how to appropriately define the spatial convolution \eqref{1.3} near the boundary of the domain while remaining analytically tractable. For this reason, we appeal to a periodic boundary condition, which requires no further modification of \eqref{1.3} near the boundary points $\{-L,L\}$. While problems \eqref{1a} and \eqref{1b} appear similar in form, it is of interest to understand exactly when and how these two formulations differ in their solution behaviours: should they be identical, it seems reasonable to choose the more tractable model depending on the goals; should they differ significantly, it is reasonable to determine \textit{when} and \textit{how} they differ, which gives insights into the validity of either case in a given context.

There exist a number of works that consider a multi-species model of the form taken in either \eqref{1a} or \eqref{1b}, see, e.g., \cite{pottslewis2016, pottslewis2019}. In these works, a linear stability analysis is performed to determine conditions sufficient for pattern formation to occur. These models are comparable in that they include a cognitive map through an additional, dynamic equation, and they also incorporate nonlocal perception. Other models with nonlocal advective operators have been studied by \cite{DucrotFuMagal2018,Giunta2021local,jungel2022nonlocal}, where some global existence results are obtained. In \cite{DucrotFuMagal2018}, fractional Sobolev spaces are utilized in a one-dimensional torus to establish a global existence result which includes the possibility of a top-hat kernel; however, the model does not incorporate a dynamic cognitive map. In \cite{Giunta2021local}, a global existence result is establish using a contraction mapping argument, but the regularity requirements of the nonlocal kernel do not include the top-hat detection function. In \cite{jungel2022nonlocal}, a global existence result is obtained for a special case of the $n$-species cross-diffusion system considered in \cite{Giunta2021local} which includes the top-hat detection function, but the kernels are assumed to be in \textit{detailed balance}, which greatly reduces the applicability for biological application. Other memory-based movement models  has been investigated in \cite{ShiShiWangYan2020,Song2022spatiotemporal}, where the cognitive map is now given by a nonlocal integral operator in \textit{time}. In such cases, the problem is a delay partial differential equation. The stability of coupled PDE-ODE models have also been studied in works such as \cite{Marciniak2017,LiMarciniakTakagiWu2017}. 

With our models at hand, the major goals of this paper are as follows. First, we seek to prove the well-posedness of models \eqref{1a} and \eqref{1b}. In particular, in Section \ref{sec:wellposedness} we prove the existence of a unique, global weak solution when the detection function $G(\cdot)$ satisfies an $L^p$-embedding type condition (see Hypotheses (H3)), which includes the discontinuous top-hat detection function. This provides an answer to Open Problems 10 and 12 found in \cite{wangsalmaniw2022}, at least for the single species case. We then shift our attention to the solution behaviour and the potential for pattern formation at a steady state. In Sections \ref{sec:stabilityanalysis}-\ref{sec:bifurcationanalysis}, we perform a robust stability and bifurcation analysis to understand the long term behaviour of the solution as it depends on parameters $d$, $\alpha$, the uptake rate $g(\cdot)$ and the kernel $G(\cdot)$. While an intuitive understanding of the relevant factors influencing pattern formation can be gleaned from a less formal linear stability analysis (see Section \ref{sec:lsa}), further care is needed for nonlocal advective operators as the essential spectrum may be non-empty. This is different from standard reaction-diffusion systems where the essential spectrum is empty. In our case, we find that the essential spectrum is entirely negative, and so has no impact on changes of stability. Section \ref{sec:top-hatanalysis} is then dedicated to a detailed case study with the top-hat detection function. To explore (perhaps subtle) differences in our formulations, we focus on three particular forms of uptake $g(\cdot)$ to better understand differences in fundamental assumptions for the function $g$. Numerical simulations using a pseudo-spectral method \cite{Giunta2021local, wangsalmaniw2022} are presented to highlight these differences.

\subsection{Preliminaries \& Hypotheses}

Denote $\mathbb{N}_0=\mathbb{N}\cup\{0\}$. Recall that the following eigenvalue problem
\begin{equation}\label{2}
\begin{cases}
-\phi''(x)=l\phi(x),\quad\quad\quad\quad\quad x\in(-L,L),\ \\
\phi(-L)=\phi(L),\ \phi'(-L)=\phi'(L), &
\end{cases}
\end{equation}
with eigenvalues and eigenfunctions
\begin{equation}\label{6}
\begin{aligned}
&l_{\pm n}=\frac{n^2\pi^2}{L^2},\ \ \phi_{\pm n}(x)=e^{\frac{\pm in\pi}{L}x}=\cos\left(\frac{n\pi}{L}x\right)\pm i\sin\left(\frac{n\pi}{L}x\right),\ n\in \mathbb{N}_0.
\end{aligned}
\end{equation}
We define the linear spaces
\begin{align*}
L^2_{per}(-L,L)=\left\{h\in L^2(-L,L): h=\sum_{n=-\infty}^{\infty}c_n\phi_n \text{ with }  \sum_{n=-\infty}^{\infty}\lvert c_n\rvert^2<\infty\right\},
\end{align*}
and 
\begin{equation*}
H^2_{per}(-L,L)=\{h\in L^2_{per}(-L,L): h''\in L^2_{per}(-L,L)\},
\end{equation*}
where 
\begin{align*}
c_n = \langle h,\phi_n \rangle = \frac{1}{2L}\int_{-L}^{L}h(x)\phi_{-n}(x) dx.
\end{align*}
We then denote by $X$ and $Y$ the spaces $H^2_{per}(-L,L) \times L^2_{per}(-L,L)$ and $L^2_{per}(-L,L)\times L^2_{per}(-L,L)$, respectively.
We always assume the following for the spatial kernel $G$:
\begin{enumerate}
\item[(H0)]\label{H0} 
$\begin{cases} 
0 \leq G(x) \in L^1_{per} (-L,L),\ G(-x) = G(x)  \text{ for all } x \in (-L,L),\cr
(2L)^{-1} \int_{-L}^L G(y)dy = 1.
\end{cases}$
\end{enumerate}

The Gaussian, exponential and top-hat detection functions all satisfy (H0) in $(-L,L)$ with the following respective scaling prefactors: $\erf(L / \sqrt{2}R)$, $1-e^{-L/R}$, and $2L$. For the stability and bifurcation analysis performed in Sections \ref{sec:stabilityanalysis}-\ref{sec:bifurcationanalysis}, we assume that the growth rates $f$ and $g$ satisfy
\begin{enumerate}
	\item[(H1)] \label{H1}
	$f(u)\in C^3([0,\infty))$, $f(0)=f(1)=0,$ $f'(0)>0$, $f'(1)<0$, $f(u)>0$ for $u\in (0,1)$ and $f(u)<0$ for $u>1.$
	\item[(H2)] \label{H2}
	$g(u)\in C^3([0,\infty))$, $g(u)>0$ on $(0,\infty)$, $g(0)=0,$ and  $g(1)=\rho >0.$
\end{enumerate}

To establish the well-posedness of the problem, we also assume in addition to (H0) that the kernel $G(\cdot)$ satisfies the following $L^p$-type estimate for any $R>0$ fixed:
\begin{enumerate}
\item[(H3)] $\norm{\overline{z}_x}_{L^p (\Omega)} \leq C \norm{z}_{L^p (\Omega)} \quad \text{ for all } z \in L^p (\Omega),\ 1 \leq p \leq \infty $.
\end{enumerate}

A prototypical example of $f(u)$ is the logistic function $f(u)=u(1-u)$, foundational in models of population growth. Biologically-motivated examples of $g(u)$ include $g(u)=\rho u$, $g(u)=\rho u^2$, and $g(u)=\rho u^2/(1+cu)$, which were discussed in the paragraph prior to Equation (\ref{2a-2}). Note that the hypotheses required in the bifurcation analysis are generally stronger than those required for well-posedness; for this reason, we state the sufficient hypotheses for the existence of a solution directly in the statement of Theorem \ref{thm:exist1a}.

Throughout this paper, we denote the null space of a linear operator $ L $ by $ \mathcal{N}(L) $, the domain of $ L $ by $ \mathcal{D}(L) $, the range of $ L $ by $ \mathcal{R}(L) $, the resolvent set of $ L $ by $ \mathcal{\rho}(L) $, and the spectrum of $ L $ by $ \mathcal{\sigma}(L) $. We always denote by $Q_T := \Omega \times (0,T) = (-L,L) \times (0,T)$.

\subsection{Statement of Main Results}

Our first result establishes the existence of a unique, nontrivial solution $(u,k)$. Due to the weak regularity assumption (H0) on the kernel $G(\cdot)$, we do not expect solutions to be classical necessarily. Denote by $h(u,k)$ the right hand side of the equation for the map $k$ in either \eqref{1a} or \eqref{1b}. We call $(u,k)$ a \textit{weak solution} to either \eqref{1a} or \eqref{1b} if, given any test function $\phi_i \in L^2 (0,T; H^1 (\Omega))$, $i=1,2$, there holds
\begin{align}
\iint_{Q_T} u_t \phi_1 {\rm d}x {\rm d}t + \iint_{Q_T}  ( d u_x + \alpha u \overline{k}_x ) (\phi_1)_x {\rm d}x {\rm d}t = \iint_{Q_T}  f(u) \phi_1 {\rm d}x {\rm d}t, \label{weaksolnu} \\
\iint_{Q_T}  k_t \phi_2 {\rm d}x {\rm d}t = \iint_{Q_T}  h(u,k) \phi_2 {\rm d}x {\rm d}t,\label{weaksolnk}
\end{align}
and the initial data is satisfied in the sense of $H^1 (\Omega)$ (in fact, the initial data will be satisfied in the sense of $C(\overline{\Omega})$ by the Sobolev embedding). We call a weak solution a \textit{global weak solution} if \eqref{weaksolnu}-\eqref{weaksolnk} holds for any $T>0$.

We have the following well-posedness result for problems \eqref{1a} and \eqref{1b}. 
\begin{theorem}\label{thm:exist1a}
Fix $T>0$, $\alpha \in \mathbb{R} \setminus \{ 0 \}$, $d,R>0$, $\mu, \beta, \kappa \geq0$, and assume that the kernel $G(\cdot)$ satisfies (H0) and (H3). Suppose that for some $\sigma \in (0,1)$, $f,g \in C^{2+\sigma} (\mathbb{R}^+)$ with $f(0) = g(0) = 0$. Assume that $f$ satisfies the bound
$$
f(z) \leq f^\prime (0) z \quad \text{ for all }\quad z \geq 0,
$$
while $g$ satisfies the bounds
\begin{align*}
g(z) &\leq N ( 1 + z^{q} ) \quad \text{ for all } \quad z\geq 0, \\
\as{g^\prime (z)} &\leq \tilde N ( 1 + z^{\tilde q} ) \quad \text{ for all } \quad z\geq 0,
\end{align*}
for some constants $N, \tilde N > 0$, $q \geq 1$ and $\tilde q \geq 0$. Finally, assume that the initial data $u_0, k_0$ satisfy
$$
0 < u_0(x), k_0(x) \in W^{1,2} (\Omega) \text{ are periodic in } \Omega.
$$
Then, there exists a unique, global weak solution $(u,k)$ solving problem \eqref{1a} in the sense of \eqref{weaksolnu}-\eqref{weaksolnk} satisfying $u \geq 0, k \geq 0$ so long as there exists $M > 0$ so that
$$
g(z) \leq M ( \mu + \beta z) \quad \text{ for all } z \geq 0.
$$
For problem \eqref{1b}, there exists a unique, global weak solution in the sense of \eqref{weaksolnu}-\eqref{weaksolnk} satisfying $u \geq 0, k \geq 0$ with no further restriction on $g(\cdot)$ other than (\ref{H2}). Moreover, in either case there holds
\begin{align}
u &\in L^{\infty} (0,T; L^p (\Omega) ) \cap C^{\sigma, \sigma/2} (\overline{Q}_T),\quad u_x,\ u_t \in L^{2} ( 0,T ; L^2 (\Omega)), \nonumber \\
k &\in L^{\infty} (0,T; L^p (\Omega) ) \cap C^{\sigma, \sigma/2} (\overline{Q}_T), \quad k_x,\ k_t \in L^{\infty} (0,T; L^2 (\Omega)) , \nonumber
\end{align}
for any $1 < p \leq \infty$, for some $\sigma \in (0,1/2)$, and the initial data is satisfied in the sense of $C(\overline{\Omega})$.
\end{theorem}

\begin{remark}\label{rmk:1.2-1}
    In the theorem above, we generally require some global polynomial growth control over the memory uptake function $g(\cdot)$. From this result, we see that problem \eqref{1a} is significantly more restrictive than problem \eqref{1b} in terms of further growth conditions on $g(\cdot)$. Indeed, the first case requires that the memory uptake behaves roughly linearly, particularly for large arguments, while the second case requires no further growth condition. From our previous discussion of biologically-motivated forms of $g(\cdot)$, we see that the forms $g(u) = \rho u$ and $g(u) = \rho u^2 / (1+cu)$ satisfy the necessary conditions for either system. On the other hand, quadratic growth $g(u) = \rho u^2$ as described in \cite{pottslewis2016} satisfies the conditions for system \eqref{1b} but not \eqref{1a}. This highlights an essential key difference between these two problems in terms of their well-posedness.
\end{remark}

For the detection function $G$ satisfying (H0), the Fourier coefficient $C_n(G)$ is defined for any $n\in \mathbb{N}$ as follows:
\begin{equation}\label{7}
C_n(G)=\frac{1}{2L}\int_{-L}^{L}e^{-\frac{in\pi }{L}y}G(y)dy=\frac{1}{2L}\int_{-L}^{L}\cos\left(\frac{n\pi }{L}y\right)G(y)dy,
\end{equation}
For $n\in\mathbb{N}$, if $C_n(G)\ne 0$, define
 \begin{equation}\label{18}
 %\begin{split}
 \alpha_{n}%&=\frac{1}{u_*h_{u*}C_{n}(G)}\left(dh_{k*}-\frac{Det(J_*)}{l_n}\right)\\
 =\frac{-(\mu+\beta)^2}{[g'(1)(\mu+\beta)-\beta\rho]C_n(G)}\left(d-\frac{f'(1)}{l_n}\right).
  %\end{split}
 \end{equation}
 Note that $C_n(G)$ could be positive or negative. Define 
\begin{equation}\label{2.24}
\begin{aligned}
     \Sigma^+:&=\{n\in \mathbb{N}: \alpha_n>0\},\ \Sigma^-:=\{n\in \mathbb{N}: \alpha_n<0\},\\
      \alpha_r:&=\ds\min_{n\in \Sigma^+} \alpha_n,\ \ \ 
     \alpha_l:=\ds\max_{n\in \Sigma^-} \alpha_n.
\end{aligned}
\end{equation}
Note that $\alpha_r>0>\alpha_l$ as $\sum_{-\infty}^{\infty} |C_n(G)|^2<\infty$. Then we have the following theorem regarding the stability of the unique constant positive steady state $U_*=(1,\rho/(\mu+\beta))$ with respect to \eqref{1a}.
\begin{theorem}\label{thm:2.8}
 Assume that assumptions (H0)-(H2) are satisfied, and let $\Sigma^+, \Sigma^-, \alpha_l, \alpha_r$ be defined as in \eqref{2.24}. Then
\begin{enumerate}
\item[(\romannumeral1)] The constant steady state solution $U_*$ is locally asymptotically stable with respect to \eqref{1a} if $\alpha_l<\alpha<\alpha_r$.
\item[(\romannumeral2)] The constant steady state solution $U_*$ is unstable with respect to \eqref{1a} if $\alpha<\alpha_l$ or $\alpha>\alpha_r$.
\end{enumerate}
\end{theorem}
Similarly we define
 \begin{equation}\label{3.31}
 \widehat{\alpha}_{n} =\frac{-(\rho+\mu+\beta)^2}{\kappa [g'(1)(\mu+\beta)-\beta\rho]C_n(G)}\left(d-\frac{f'(1)}{l_n}\right),
 \end{equation}
and 
\begin{equation}\label{3.33}
\begin{aligned}
     \widehat{\Sigma}^+:&=\{n\in \mathbb{N}: \widehat{\alpha}_n>0\},\ \widehat{\Sigma}^-:=\{n\in \mathbb{N}: \widehat{\alpha}_n<0\},\\
      \widehat{\alpha}_r:&=\ds\min_{n\in \widehat{\Sigma}^+} \widehat{\alpha}_n,\ \ \ 
     \widehat{\alpha_l}:=\ds\max_{n\in \widehat{\Sigma}^-} \widehat{\alpha}_n.
\end{aligned}
\end{equation}
Then we have the stability results for the unique constant positive steady state $\widehat{U_*}=(1,\rho\kappa/(\rho+\mu+\beta))$ with respect to \eqref{1b}.
\begin{theorem}\label{thm:2.9}
 Assume that assumptions (H0)-(H2) are satisfied, and let $\widehat{\Sigma}^+, \widehat{\Sigma}^-, \widehat{\alpha}_l, \widehat{\alpha}_r$ be defined as in \eqref{3.33}. Then
\begin{enumerate}
\item[(\romannumeral1)] The constant steady state solution $\widehat{U}_*$ is locally asymptotically stable with respect to \eqref{1b} if $\widehat{\alpha}_l<\alpha<\widehat{\alpha}_r$.
\item[(\romannumeral2)] The constant steady state solution $\widehat{U}_*$ is unstable with respect to \eqref{1b} if $\alpha<\widehat{\alpha}_l$ or $\alpha>\widehat{\alpha}_r$.
\end{enumerate}
\end{theorem}

The quantities $\alpha_n$ ($\widehat{\alpha}_n$) defined in \eqref{18} (\eqref{3.31}) are the critical parameter values such that the stability of the spatially-constant steady state changes, and they are also bifurcation points for \eqref{1a} (\eqref{1b}) where spatially non-homogeneous steady state solutions bifurcate from the constant ones as found in the following Theorems.
\begin{theorem}\label{thm:3.4}
Assume that assumptions (H0)-(H2) are satisfied, and $n\in {\mathbb N}$ such that $C_n(G)\ne 0$. Then near $(\alpha,U)=(\alpha_n,U_*)$, problem \eqref{1a} has a line of trivial solutions $\Gamma_0:=\{(\alpha,U_*):\alpha\in {\mathbb R}\}$ and a family of  non-constant steady state solutions bifurcating from $\Gamma_0$  at $\alpha=\alpha_n$ in a  form of 
	\begin{equation}
	\Gamma_n : =\{(\alpha_n(s),u_n(s,\cdot),k_n(s,\cdot)) : -\delta<s<\delta\}
	\end{equation}
	with
	\begin{equation}
	\begin{cases}
	\alpha_n(s)=\alpha_{n}+\alpha_n'(0)s+o(s),\\
	\ds u_n(s,x)=1+s\cos\left(\dfrac{n\pi}{L}x\right)+s^2 z_{1n}(s,x),\\
	\ds k_n(s,x)=\dfrac{\rho}{\mu+\beta}+s\dfrac{g'(1)(\mu+\beta)-\beta \rho}{(\mu+\beta)^2}\cos\left(\dfrac{n\pi}{L}x\right)+s^2 z_{2n}(s,x),
	\end{cases}
	\end{equation}
where $z_n(s)=(z_{1n}(s,\cdot),z_{2n}(s,\cdot))$ satisfies
$\ds\lim\limits_{s\rightarrow0}\lVert z_n(s)\rVert=0$.
Moreover the set of steady state solutions of \eqref{1a} near $(\alpha_n,U_*)$ consists precisely of the curves $\Gamma_0$ and $\Gamma_n$.
\end{theorem}

\begin{theorem}\label{thm:3.5}
Assume that assumptions (H0)-(H2) are satisfied, and $n\in {\mathbb N}$ such that $C_n(G)\ne 0$. Then near $(\widehat{\alpha},\widehat{U})=(\widehat{\alpha}_n,\widehat{U}_*)$, equation \eqref{1b} has a line of trivial solutions $\widehat{\Gamma}_0:=\{(\alpha,\widehat{U}_*):\alpha\in {\mathbb R}\}$ and a family of non-constant steady state solutions bifurcating from $\widehat{\Gamma}_0$ at $\alpha=\widehat{\alpha}_n$ in a form of 
	\begin{equation}
	\widehat{\Gamma}_n : =\{(\widehat{\alpha}_n(s),\widehat{u}_n(s,\cdot),\widehat{k}_n(s,\cdot)) : -\delta<s<\delta\}
	\end{equation}
	with
	\begin{equation}
	\begin{cases}
	\widehat{\alpha}_n(s)=\widehat{\alpha}_{n}+\widehat{\alpha}_n'(0)s+o(s),\\
	\ds \widehat{u}_n(s,x)=1+s\cos\left(\dfrac{n\pi}{L}x\right)+s^2 \widehat{z}_{1n}(s,x),\\
	\ds \widehat{k}_n(s,x)=\dfrac{\rho}{\rho+\mu+\beta}+\kappa s\dfrac{g'(1)(\mu+\beta)-\beta \rho}{(\mu+\beta+\rho))^2}\cos\left(\dfrac{n\pi}{L}x\right)+s^2 \widehat{z}_{2n}(s,x),
	\end{cases}
	\end{equation}
	where $\widehat{z}_n(s)=(\widehat{z}_{1n}(s,\cdot),\widehat{z}_{2n}(s,\cdot))$ satisfies
$\ds\lim\limits_{s\rightarrow0}\lVert \widehat{z}_n(s)\rVert=0$.
Moreover the set of steady state solutions of \eqref{1b} near $(\widehat{\alpha}_n,\widehat{U}_*)$ consists precisely of the curves $\widehat{\Gamma}_0$ and $\widehat{\Gamma}_n$.
\end{theorem}

We also classify the nature of the bifurcation at these critical values, see Theorems \ref{thm:4.6} and \ref{thm:4.7}. Together, these results show that the central quantity governing spontaneous pattern formation is the advective strength towards or away from memorised areas, encapsulated in $\alpha$.  Since $\alpha_l<0<\alpha_r$ holds necessarily from the assumptions made, the key driver of pattern formation is that $\alpha$ is of sufficient magnitude.  If $\alpha$ is negative, then we have attraction towards remembered areas, similar to many nonlocal models of biological aggregation (e.g. \cite{carrillo2019aggregation}).  Some examples are found in Figures \ref{fig:1b}, \ref{fig:2b} and \ref{fig:3b}.  On the flip-side, positive $\alpha$ indicates repulsion from remembered areas and leads to patterns such as in Figures \ref{fig:1c}, \ref{fig:2c} and \ref{fig:3c}. Interestingly, there is a lack of symmetry in the sense that $-\alpha_l \neq \alpha_r$ in general. In fact, $\as{\alpha_l}$ and $\alpha_r$ do not even remain ordered! This can be seen in Figures \ref{fig:1}, \ref{fig:2} and \ref{fig:3}.

Moreover, Theorems \ref{thm:3.4} and \ref{thm:3.5} show that, close to the bifurcation point, the steady state consists of a single cosine wave, and Theorems \ref{thm:4.6} and \ref{thm:4.7} give conditions on their stability.  An example of the stable case is shown in Figures \ref{fig:2b}-\ref{fig:2c}, where a cosine wave emerges just beyond the bifurcation point, but diverges from this description as the bifurcation parameter is increased further.  For the unstable case we do not see the small-amplitude cosine wave at all as the bifurcation threshold is crossed.  Rather, the  solution jumps to a higher amplitude pattern as found in Figure \ref{fig:1c}.

\section{Well-posedness}\label{sec:wellposedness}

In this section we prove the existence of a unique global solution to a problem more general than system \eqref{1a} or \eqref{1b} for detection functions $G(\cdot)$ satisfying (H0) and (H3), which includes the top-hat kernel. The restrictions on the growth terms $f(\cdot)$ and $g(\cdot)$ are compatible with the choices made in Section \ref{sec:top-hatanalysis}. The challenge is in the treatment of the (potentially) discontinuous kernel appearing inside the nonlocal advection term. To overcome these difficulties, we abuse a useful `embedding' property of the top-hat kernel in one spatial dimension. This allows one to obtain {\it a priori} estimates on the solution $k(x,t)$ from which we obtain appropriate uniform bounds on $u(x,t)$ and higher derivatives. We begin with some preliminary estimates.

\subsection{Preliminary Estimates}

In general, we assume $G (\cdot)$ satisfies (H3); we first show this holds for the top-hat kernel.
\begin{lemma}\label{kernelbound}
Let $1 \leq p \leq \infty$ and fix $T>0$. Suppose $z(\cdot,t) \in L^p (\Omega)$ is periodic in $\Omega$ for almost every $t \in [0,T]$. Denote by $\overline{z}_x(\cdot,t)$ the spatial convolution \eqref{1.3} of $z_x$ with the top-hat detection function \eqref{detectionkernelT}. Then, for almost every $t \in [0,T]$ there holds
\begin{align}\label{bnd1.1}
\norm{\overline{z}_x (
\cdot,t)}_{L^p (\Omega)} &\leq  \frac{1}{R} \norm{z(\cdot,t)}_{L^p (\Omega)}.
\end{align}
In particular, we have that
\begin{align}\label{bnd1.2}
\esssup_{t \in (0,T)}\norm{\overline{z}_x (\cdot,t)}_{L^p (\Omega)} &\leq R^{-1} \esssup_{t \in (0,T)} \norm{z (\cdot,t)}_{L^p (\Omega)},
\end{align}
for any $T>0$ fixed. 
\end{lemma}
\begin{remark}
    If $z(\cdot,t)$ is continuous, we may replace ``almost every" with ``every" and $\esssup$ with $\sup$. This will be the case in the forthcoming results. 
\end{remark}
\begin{proof}
First, we drop the dependence on $t$ for notational brevity. The result essentially follows from an elementary inequality and the fact that 
$$
\overline{z}_x = \frac{z(x+R) - z(x-R)}{2R}
$$ 
when $G(\cdot)$ is the top hat detection function. Consequently,
\begin{align}\label{Y1.1}
\norm{\overline{z}_x}_{L^p (\Omega)} ^p &= \frac{1}{(2R)^p}  \int_\Omega \as{z(x+R) - z(x-R)}^p {\rm d}x \nonumber \\
&\leq\frac{2^{p-1}}{(2R)^p} \int_\Omega \left( \as{z(x+R)}^p + \as{z(x-R)}^p \right) {\rm d}x \leq  \frac{1}{R^p} \norm{z}_{L^p (\Omega)} ^p ,
\end{align}
where we have used the periodicity of $z$ in $\Omega$. This proves \eqref{bnd1.1}. Taking the $p^{\text{th}}$ roots of both sides followed by the supremum over $t \in (0,T)$ yields \eqref{bnd1.2}.
\end{proof}
\begin{remark}
    The same $L^p$-type estimate holds for the Gaussian and exponential detection functions \eqref{detectionkernelG}-\eqref{detectionkernelE} as well. These cases are easier since the kernels themselves are appropriately differentiable and bounded. We omit the details.
\end{remark}

Next, we obtain $L^p (\Omega)$ bounds on a function $k(x,t)$ when $k$ solves a linear, first order differential equation for each $x \in \Omega$.
\begin{lemma}\label{ODEbounds1}
Let $0 \lneqq w(x,t) \in C^{1,1} (\overline{Q}_T) \cap L^{1,1} (Q_T)$ be periodic in $\Omega$ for all $t \in (0,T)$ and assume $1 < p \leq \infty$. For each $x \in \Omega$, let $k(x,\cdot)$ solve the ordinary differential equation
\begin{equation}\label{kProto}
\frac{d k}{d t} = g_1(w) - g_2(w) k
\end{equation}
where $g_1, g_2 \in C^1 (\mathbb{R}^+)$ are nonnegative and $k(x,0) = k_0 (x) \in W^{1,2} (\Omega)$. Then, if there exists $M>0$ such that
\begin{equation}\label{g-bound-1}
g_1(z) \leq M g_2 (z) \quad \text{ for all } z \geq 0,
\end{equation}
there holds
\begin{equation}\label{kLinfest}
\sup_{t \in [0,T]} \norm{k (\cdot,t)}_{L^\infty (\Omega)} \leq M + \norm{k_0}_{L^\infty (\Omega)}
\end{equation}
\end{lemma}
\begin{proof}
First, note that by solving the differential equation directly, $k (x,\cdot) \in C^{1} ( [ 0, T])$. By the smoothness of $g_i(\cdot)$, $i=1,2$ and the boundedness of $w(x,t)$ in $Q_T$, $k(\cdot,t) \in L^p (\Omega)$ for any $p > 1$, for all $t \in (0,T)$. Taking the time derivative of $ \tfrac{1}{p}\norm{k (\cdot,t)}_{L^p (\Omega)}^p$ gives
\begin{align}\label{Y1.3}
\frac{1}{p} \frac{d}{dt} \int_\Omega k^p{\rm d}x &= \int_\Omega k^{p-1} g_1(w) {\rm d}x - \int_\Omega k^p g_2(w) {\rm d}x.
\end{align}
We now apply Young's inequality. To this end, we carefully rewrite as
\begin{align}
k^{p-1} g_1(w) &= k^{p-1} g_2(w)^{(p-1)/p} 
 \cdot \frac{g_1(w)}{g_2(w)^{(p-1)/p}} \nonumber \\
&\leq \frac{1}{p_1}\left( k^{p-1} g_2(w)^{(p-1)/p} \right)^{p_1} + \frac{1}{q_1}\left( \frac{g(w)}{g_2(w)^{(p-1)/p}} \right)^{q_1},
\end{align}
where $p_1,q_1 > 1$ satisfy $p_1 ^{-1} + q_1 ^{-1} = 1$. Choosing $p_1 = p/(p-1)$ and $q_1 = p$, equation \eqref{Y1.3} then becomes
\begin{align}\label{1.1.pfchange}
 \frac{d}{dt}\int_\Omega k^p {\rm d}x &\leq \int_\Omega \left( \frac{g_1(w)}{g_2(w)} \right)^p g_2(w) {\rm d}x
\end{align}
Using bound \eqref{g-bound-1} we then have
\begin{align}
 \frac{d}{dt} \int_\Omega k^p{\rm d}x  &\leq M^{p} \norm{g_2 (w(\cdot,t))}_{L^1 (\Omega)} .
\end{align}
Integrating both sides from $0$ to $t$ yields
\begin{align}
\norm{k(\cdot,t)}_{L^p (\Omega)}^p &\leq M^{p} \norm{g_2 (w)}_{L^{1,1} (Q_T)} + \norm{k_0}_{L^p (\Omega)}^p .
\end{align}
Taking $p^{th}$ roots of both sides and sending $p \to \infty$ leaves
\begin{align}
\norm{k(\cdot,t)}_{L^\infty (\Omega)} &\leq M + \norm{k_0}_{L^\infty (\Omega)} .
\end{align}
Taking the supremum over $t \in (0,T)$ yields \eqref{kLinfest}, completing the proof.
\end{proof}
We also highlight the following properties of $k(x,t)$ inherited by the function $w(x,t)$, a simple consequence of solving the ordinary differential equation.
\begin{proposition}
Suppose $k(x,\cdot)$ solves \eqref{kProto} with $w \in C^{1,1} (\overline{Q}_T)$ periodic in $\Omega$. Then, $k(x,t) \in C^{1,2} (\overline{Q}_T)$. Moreover, $k(\cdot,t)$ is periodic in $\Omega$ for all $t >0$.
\end{proposition}
Finally, we obtain $L^p$ estiates on the time/space derivatives $k_t$ and $k_x$.
\begin{theorem}\label{ODEbounds2}
Assume the same conditions as in Lemma \ref{ODEbounds1} hold. Assume also that there exists $N>0$ and $q \geq 1$ fixed so that
\begin{equation}\label{g-bound-2}
g_2(z) \leq N ( 1 + z^q ) \quad \text{for all } z \geq 0.
\end{equation}
Assume in addition that $w_x \in L^{2,2} (Q_T)$ and $\sup_{t \in (0,T)} \norm{w(\cdot,t)}_{L^p (\Omega)} < \infty$ for all $p \geq 1$. Then there holds
\begin{align}\label{bnd3.1}
&\sup_{t \in (0,T)} \norm{k_t (\cdot,t)}_{L^{p} (\Omega)} \leq %\nonumber \\
%& 
4 N \left( M + \sup_{t \in (0,T)} \norm{k (\cdot,t)}_{L^\infty (\Omega)} \right) \left( \as{\Omega}^{1/p} + \sup_{t \in (0,T)} \norm{w(\cdot,t)}_{L^{pq} (\Omega)}^q \right)
\end{align}
for any $p \in (1,\infty)$. Moreover, if for some $\tilde N > 0$, $\tilde q \geq 0$ fixed we have that
\begin{align}\label{g-bound-3}
    \as{g_1 ^\prime (z)}, \as{g_2^\prime (z)} \leq \tilde N ( 1 + z^{\tilde q} ) \quad \text{ for all } z \geq 0,
\end{align}
then there exists a constant $C>0$ depending only on $T$, $\tilde N$, and $\tilde q$ so that
\begin{align}\label{bnd3.2}
     \sup _{t \in (0,T)} \norm{k_x (\cdot,t)}_{L^p (\Omega)} &\leq C \sup_{t \in (0,T)} \norm{(1 + w^{\tilde q})(\cdot,t)}_{L^{2p/(2-p)} (\Omega)} \norm{w_x}_{L^{2,2} (Q_T)}  %\nonumber \\
     %& 
     + \norm{(k_0)_x}_{L^p (\Omega)} ,
\end{align}
for any $p \in (1,2)$ whenever $\tilde q > 0$, and any $p \in (1,2]$ whenever $\tilde q = 0$. 
\end{theorem}
\begin{proof}
Integrating $\as{k_t}^p$ over $\Omega$, applying an elementary inequality, and using \eqref{g-bound-1} yields
\begin{align}\label{Y1.179}
\int_\Omega \as{k_t} ^p{\rm d}x &= \int_\Omega \as{ g_1(w) -g_2(w) k}^p{\rm d}x \nonumber \\
&\leq 2^{p-1} \int_\Omega\left(  \as{g_1(w)}^p + \as{g_2(w)}^p \as{k}^p\right) {\rm d}x \nonumber.\\
&\leq 2^{p-1} \int_\Omega \left( M^p + \as{k}^p  \right) \as{g_2(w)}^p {\rm d}x
\end{align}
Estimating further, we use the bound on $k$ along with bound \eqref{g-bound-2} and the same elementary inequality to see that
\begin{align*}
\int_\Omega \as{k_t} ^p{\rm d}x  &\leq 2^{p-1} \left( M^p + \norm{k(\cdot,t)}_{L^\infty (\Omega)}^p \right) \int_\Omega \as{g_2(w)}^p {\rm d}x\\
&\leq 4^{p-1} N^p \left( M^p + \norm{k(\cdot,t)}_{L^\infty (\Omega)}^p \right) \int_\Omega \left( 1 + w^{pq}\right) {\rm d}x \nonumber \\
&\leq  4^{p} N^p \left( M^p + \norm{k(\cdot,t)}_{L^\infty (\Omega)}^p \right)  \left(\as{\Omega} + \norm{w(\cdot,t)}^{pq}_{L^{pq} (\Omega)} \right).
\end{align*}
Thus, we take the $p^{th}$ roots of both sides followed by the supremum over $t \in (0,T)$ to obtain \eqref{bnd3.1}. This completes the first part of the proof.

Next, we obtain $L^p$ bounds on $k_x$ for any $p \in (1,2)$. Solving the ordinary differential equation, we may compute $k_x (x,t)$ as follows:
\begin{align}
    k_x (x,t) &= \frac{\partial}{\partial x} \left( \int_0 ^t e^{- \int_s ^t g_2(w) {\rm d} \xi} g_1(w) {\rm d}s + k_0 (x)  \right) \nonumber \\
    &= \int_0 ^t e^{- \int_s ^t g_2(w) {\rm d} \xi } \left( g_1^\prime (w) w_x - \int_s ^t g_2^\prime (w) w_x d \xi  \right) {\rm d}s + (k_0)_x.
\end{align}
Therefore, estimating crudely and using bound \eqref{g-bound-3} there holds
\begin{align}
    \as{k_x} &\leq \int_0 ^T \left( \as{g_1^\prime (w)} \as{w_x} + \int_0 ^T \as{g^\prime _2 (w)}\as{w_x} {\rm d} \xi \right) {\rm d} s + \as{( k_0 ) _x} \nonumber \\
    &\leq \int_0 ^T \left( \as{g_1 ^\prime (w)} + T \as{g^\prime _2 (w)}  \right) \as{w_x} {\rm d}s + \as{(k_0)_x} \nonumber \\
    &\leq \tilde N (1 + T) \int_0 ^T \left( 1 + w^{\tilde q}\right) \as{w_x} {\rm d} s + \as{(k_0)_x}.
\end{align}
Raising both sides to the power $p$, integrating over $\Omega$ and estimating via an elementary application of H{\" o}lder's inequality in the temporal domain yields
\begin{align}\label{Y.3.5.7}
     \int_\Omega \as{k_x}^p {\rm d}x &\leq \tilde N ^p (1 + T)^p \int_\Omega \left( \int_0 ^T (1 + w^{\tilde q} ) \as{w_x} {\rm d} s \right)^p {\rm d}x + \norm{(k_0)_x}^p_{L^p (\Omega)} \nonumber \\
     &\leq \tilde N ^p (1 + T)^p T^{p-1}\iint_{Q_T} (1 + w^{\tilde q} )^p \as{w_x}^p {\rm d}x {\rm d} s  + \norm{(k_0)_x}^p_{L^p (\Omega)}.
\end{align}
We now apply H{\" o}lder's inequality in the spatial domain as follows:
$$
\int_\Omega (1 + w^{\tilde q})^p \as{w_x}^p {\rm d}x \leq \left( \int_\Omega \as{w_x}^{p p_1} {\rm d}x\right)^{1/p_1}  \left( \int_\Omega (1 + w ^{\tilde q})^{p q_1} {\rm d}x\right)^{1/q_1},
$$
where we again choose $p_1 = 2/p > 1$ so that $q_1 = 2/(2-p) >1$. Simplifying and taking the supremum over $t \in (0,T)$ for the lower order term, we find
$$
\int_\Omega (1 + w^{\tilde q})^p \as{w_x}^p {\rm d}x \leq \norm{w_x (\cdot,t)}_{L^2 (\Omega)}^p \sup_{t \in (0,T)} \norm{(1 + w^{\tilde q})(\cdot,t)}_{L^{2p/(2-p)}(\Omega)}^p 
$$
\eqref{Y.3.5.7} then becomes
\begin{align}\label{Y.3.5.72}
 \int_\Omega \as{k_x}^p {\rm d}x &\leq \tilde N ^p (1 + T)^p T^{p-1}\sup_{t \in (0,T)} \norm{(1 + w^{\tilde q})(\cdot,t)}_{L^{2p/(2-p)} (\Omega)}^p \int_0^T \norm{w_x (\cdot,s)}_{L^2 (\Omega)}^p {\rm d}s %\nonumber \\
% &
+ \norm{(k_0)_x}_{L^p (\Omega)}.
\end{align}
Finally, we apply H{\" o}lder's inequality in the temporal domain once more as follows:
$$
\int_0 ^T \norm{w_x (\cdot,s)}_{L^2 (\Omega)}^p {\rm d}s \leq T^{1 - p/2} \norm{w_x}_{L^{2,2} (Q_T)}^p ,
$$
whence \eqref{Y.3.5.72} becomes
\begin{align}
     \int_\Omega \as{k_x}^p {\rm d}x &\leq \tilde N ^p (1 + T)^p T^{p/2}\sup_{t \in (0,T)} \norm{(1 + w^{\tilde q})(\cdot,t)}_{L^{2p/(2-p)} (\Omega)}^p \norm{w_x}_{L^{2,2} (Q_T)}^p  %\nonumber %\\
     %& 
     + \norm{(k_0)_x}_{L^p (\Omega)}^p .
\end{align}
Taking the $p^{th}$ roots of both sides followed by the supremum over $t \in (0,T)$ yields \eqref{bnd3.2}, valid for any $p \in (1,2)$. Finally, if $\tilde q = 0$, the dependence on $w^{\tilde q}$ vanishes and the bound holds for $p=2$, completing the proof.
\end{proof}

\subsection{Existence \& Uniqueness}

We are now prepared to use these preliminary estimates to prove the existence of a weak solution. Much of the heavy lifting is now complete. What remains is to construct an appropriate sequence of approximate solutions and use our previously obtained estimates to extract a convergent subsequence.

\begin{proof}[Proof of Theorem \ref{thm:exist1a}]
We prove the existence of a global weak solution to the following general system subject to periodic boundary conditions:
\begin{align}
\label{eq:general}
\begin{cases}
u_t = d u_{xx} + \alpha ( u \overline{k}_x )_x + f(u), \quad &\text{ in } Q_T, \cr
k_t = g_1(u) - g_2 (u)k \quad &\text{ in } Q_T,
\end{cases}
\end{align}
where systems \eqref{1a} and \eqref{1b} are obtained by choosing $g_1(u) := g(u)$, $g_2(u) := \mu + \beta u$ or $g_1(u):= \kappa g(u)$, $g_2(u) := \mu + \beta u + g(u)$, respectively.
First we construct a sequence of approximate solutions via the following iteration scheme:
\begin{align}
\label{eq:iter}
\begin{cases}
(u_n)_t = d (u_n)_{xx} + \alpha ( u_n (\overline{k_n})_x )_x + f(u_n), \quad &\text{ in } Q_T, \cr
(k_n)_t = g_1(u_{n-1}) - g_2 (u_{n-1})k_n \quad &\text{ in } Q_T,
\end{cases}
\end{align}
for $n \geq 2$, where we choose $(u_n(x,0), k_n (x,0) ) = ( u_0 (x), k_0(x))$ for each $n$. Note carefully that $u_n = u_n (x,t)$ is a function defined in $\overline{Q}_T$ for all $n \geq 1$, whereas $u_0 = u_0 (x)$ denotes the fixed initial data of the original problem. The same holds for $\{ k_n \}_{n \geq 2}$, each of which are defined over $\overline{Q}_T$; notice also that we do not refer to $k_1 (x,t)$ as we require only $u_1(x,t)$ to initiate.

Through this construction, we generate a sequence of solutions $ \{ ( u_n, k_n ) \}_{n \geq 2}$. More precisely, we choose the initial iterate $0 < u_1 (x,t) \in C^{2+\sigma, 1+\sigma/2} ( \overline{Q}_T )$ for some $\sigma \in (0,1)$. By solving differential equation for $k_2 (x,t)$, the dependence of $k_2 (x,t)$ on the sufficiently regular functions $u_1 (x,t)$ and $g_i(\cdot)$, $i=1,2$, ensures that $k_2 \in C^{2 + \sigma, 1 + \sigma / 2} ( \overline{Q}_T )$ for some (possibly smaller) $\sigma \in (0,1)$ as well. Furthermore, the positivity of $u_1$ and $u_1 (x,0) = u_0(x)$ ensures that $k_2 \gneqq 0$ in $Q_T$. Then, the existence of a nonnegative, nontrivial solution $u_2 (x,t) \in C^{2+\sigma, 1+\sigma/2}(\overline{Q}_T)$ follows from the classical theory of parabolic equations since it is a second order, semi-linear parabolic equation with H\"{o}lder continuous coefficients \cite{ladyzhenskaia1968linear}. Therefore, there exists a nonnegative, nontrivial classical solution pair $(u_2, k_2)$ each belonging to $C^{2+\sigma, 1+ \sigma/2} (\overline{Q}_T )$ for some $\sigma \in (0,1)$. One may then proceed inductively, proving the existence of a nonnegative, nontrivial classical solution pair $(u_n, k_n)$ for any $n \geq 3$ using the regularity of the previous iterate $u_{n-1} (x,t)$. We now seek uniform bounds in a weaker setting.

To this end, fix $n\geq2$. It is easy to obtain $L^1$-bounds on $u_n$ as follows:
\begin{align}
    \frac{{\rm d}}{{\rm d}t} \int_\Omega u_n {\rm d}x = \int_\Omega f(u_n) {\rm d}x \leq f^\prime (0) \int_\Omega u_n {\rm d}x,
\end{align}
where we have integrated by parts, applied the boundary conditions, and used the assumed bound $f(z) \leq f^\prime (0) z$ for all $z \geq0$. Gr{\"o}nwall's inequality implies that
$$
\norm{u_n (\cdot, t)}_{L^1 (\Omega)} \leq e^{f^\prime (0) T} \norm{u_0 }_{L^1 (\Omega)}. 
$$
Thus, integrating over $(0,T)$ yields
\begin{align}\label{L1boundu}
    \norm{u_n}_{L^{1,1} (Q_T)} \leq T e^{f^\prime (0) T} \norm{u_0}_{L^1 (\Omega)},
\end{align}
and so $\{ u_n \}_{n \geq 1}$ is uniformly bounded in $L^{1,1} (Q_T)$ for any $T>0$ fixed. 

Return now to the equation for $k_n$. The smoothness of the iterates $u_n$ allows us to apply Lemma \ref{ODEbounds1}, giving us
\begin{align}\label{Y1.2.0}
    \sup_{t \in (0,T)} \norm{k_n (\cdot,t)}_{L^\infty (\Omega)} &\leq M + \norm{k_0}_{L^\infty (\Omega)}.
\end{align}
Note that $k_0 \in W^{1,2} (\Omega) \Rightarrow k_0 \in L^\infty (\Omega)$ by the Sobolev embedding. Then, Lemma \ref{kernelbound} paired with estimate \eqref{Y1.2.0} implies that
\begin{align}\label{Y1.2.33}
    \sup_{t \in (0,T)} \norm{(\overline{k_n})_x (\cdot, t)}_{L^\infty(\Omega)} &\leq R^{-1} \sup_{t \in (0,T)} \norm{k_n (\cdot,t)}_{L^\infty (\Omega)} \nonumber \\
    &\leq R^{-1} \left( M + \norm{k_0}_{L^\infty (\Omega)} \right) =: C_1
\end{align}
Now we seek $L^p$-bounds on the iterates $u_n$. Fix $p \geq 2$. Taking the time derivative of $\tfrac{1}{p} \norm{u_n(\cdot,t)}_{L^{p} (\Omega)} ^{p}$, integrating by parts yields and using the bound for $f(\cdot)$ yields
\begin{align}\label{Y1.5}
\frac{1}{p}\frac{d}{dt} \int_\Omega u_n^{p} {\rm d}x&= \int_\Omega u_n^{p-1} ( d (u_n)_x + \alpha u_n (\overline{k_n})_x)_x{\rm d}x + \int_\Omega u_n^{p-1} f(u_n){\rm d}x \nonumber \\ 
&\leq f^\prime(0) \int_\Omega u_n ^p {\rm d}x - d (p-1) \int_\Omega u_n^{p-2} \magg{(u_n)_x}{\rm d}x \nonumber \\
&\ + \as{\alpha}  (p-1) \int_\Omega u_n^{p-1} \as{(u_n)_x} \as{(\overline{k_n})_x} {\rm d}x.
\end{align}
We now use \eqref{Y1.2.33} and Cauchy's inequality with epsilon to control the third term by the second term on the right hand side of \eqref{Y1.5}. To this end, we estimate
\begin{align}
    \as{\alpha} u_n ^{p-1} \as{(u_n)_x} \as{(\overline{k}_n)_x} &\leq \as{\alpha} C_1 u_n ^{(p-2)/2} \as{(u_n)_x} \as{u_n}^{p/2} \nonumber \\
    &\leq \as{\alpha} C_1 \left( \frac{\varepsilon}{2} u_n ^p +  \frac{1}{2 \varepsilon} u_n ^{p-2} \as{(u_n)_x}^2 \right),
\end{align}
where we choose $\varepsilon = d^{-1} \as{\alpha} C_1$. Paired with \eqref{Y1.5}, this leaves
\begin{align}\label{Y1.52}
    \frac{1}{p}\frac{d}{dt} \int_\Omega u_n^{p} {\rm d}x&\leq -\frac{d}{2} \int_\Omega u_n ^{p-2} \as{(u_n)_x}^2 {\rm d}x %\nonumber \\
   % &\quad 
   +  ( f^\prime (0) + C_1^2 \alpha^2 d^{-1} (p-1) ) \int_\Omega u_n ^p {\rm d}x
\end{align}
Therefore, dropping the negative term and applying Gr{\"o}nwall's inequality yields
$$
    \norm{u_n (\cdot,t)}^p_{L^p (\Omega)} \leq e^{p(f^\prime (0) + C_1^2 \alpha^2 d^{-1} (p-1))T} \norm{u_0}^p _{L^p (\Omega)}.
$$
Taking $p^{\text{th}}$ roots followed by the supremum over $t \in (0,T)$ yields the estimate
\begin{align}\label{Lpuest}
 \sup_{t \in (0,T)} \norm{u_n (\cdot,t)}_{L^p ( \Omega)} &\leq e^{(f^\prime (0) + C_1 ^2\alpha^2 d^{-1} (p-1))T} \norm{u_0}_{L^p (\Omega)}  \nonumber \\
&=: C_2  ,
\end{align}
noting that the exponent depends critically on $p$. Next, we return to \eqref{Y1.52} for the case $p=2$. Upon rearrangement, we apply estimate \eqref{Lpuest} to obtain
\begin{align}
\frac{1}{2}\frac{d}{dt} \int_\Omega u_n^{2}{\rm d}x + \frac{d}{2} \int_\Omega \magg{(u_n)_x}{\rm d}x &\leq \left(f^\prime (0) + C_1 ^2 \alpha^2 d^{-1} \right) \int_\Omega u_n^2 {\rm d}x \nonumber \\
&\leq C_2^2  \left(f^\prime (0) + C_1 ^2 \alpha^2 d^{-1} \right) \nonumber \\
&=: C_3.
\end{align}
Integrating both sides from $0$ to $T$ yields
\begin{align}
    \frac{1}{2} \left( \norm{u_n (\cdot,T)}_{L^2 (\Omega)}^2 - \norm{u_0}_{L^2 (\Omega)}^2 + d \norm{(u_n)_x}_{L^{2,2} (Q_T)}^2  \right) \leq C_3 T.
\end{align}
Ignoring the positive term on the left hand side, we extract the desired estimate for $(u_n)_x$:
\begin{align}\label{uxest}
    \norm{(u_n)_x}_{L^{2,2} (Q_T)}^2 &\leq d^{-1} \left( 2 C_3 T +  \norm{u_0}_{L^2 (\Omega)} \right)=: C_4^2.
\end{align}

We now immediately have the boundedness of $(k_n)_x$ and $(k_n)_t$ for any $n \geq 2$ in some $L^p$ spaces. Indeed, by Theorem \ref{ODEbounds2},  estimates  \eqref{Y1.2.0}, \eqref{Lpuest} and \eqref{uxest} imply the existence of a constant $C_5$, independent of $n$, such that
\begin{align}\label{knbounds1}
    \sup_{t \in (0,T)} \norm{(k_n)_t (\cdot,t)}_{L^p (\Omega)},\ \sup_{t \in (0,T)} \norm{(k_n)_x (\cdot,t)}_{L^p (\Omega)} \leq C_5,
\end{align}
for any $p \in (1,2)$.

We now appeal to standard $L^p$-estimates for parabolic equations and the Sobolev embedding to improve our estimates on $u_n$. If we expand the equation for $u_n$ it reads
$$
(u_n)_t - d (u_n)_{xx} = \alpha \left( \overline{(k_n)}_x (u_n)_x + u_n \overline{(k_n)}_{xx} \right) + f(u_n).
$$
Obviously, $f(z) \leq f^\prime (0) z$ and $u_n \in L^{p,p}(Q_T)$ for any $p \geq 1$ gives us that $f(u_n) \in L^{p,p}(Q_T)$ as well. Then, $L^p$-estimates for strong solutions (see, e.g., \cite{Wang2003}) ensures that there holds
\begin{align}
\norm{u_n}_{W^{2,1}_r (Q_T)} &\leq C \left( \norm{ \overline{(k_n)}_x (u_n)_x}_{L^r (Q_T)} + \norm{ u_n \overline{(k_n)}_{xx}}_{L^r (Q_T)} + \norm{u_n}_{L^r (Q_T)} \right),
\end{align}
for some $C>0$, for any $r>1$. Choosing $r \in (1,p)$, H\"{o}lder's inequality gives
\begin{align}
\norm{u_n \overline{(k_n)}_{xx}}_{L^r (Q_T)} \leq \norm{u_n}_{L^{pr/(p-r)} (Q_T)} \norm{\overline{(k_n)}_{xx}}_{L^p (Q_T)}.
\end{align}
By Lemma \eqref{kernelbound}, the bound \eqref{knbounds1} and \eqref{Lpuest}, we may further estimate as
\begin{align}
\norm{u_n}_{L^{pr/(p-r)} (Q_T)} \norm{\overline{(k_n)}_{xx}}_{L^p (Q_T)} &\leq R^{-1} C_2 \sup_{t \in (0,T)} \norm{(k_n)_x (\cdot,t)}_{L^p (\Omega)} \nonumber \\
&\leq R^{-1} C_2 C_5 =: C_6,
\end{align}
for any $r \in (1,p)$, where $C_6$ does not depend on $n$. Similarly, there holds
$$
\norm{\overline{(k_n)}_x (u_n)_x}_{L^r (Q_T)} \leq C_7,
$$
where $C_7$ does not depend on $n$. Hence,
$$
\norm{u_n}_{W^{2,1}_r (Q_T)} \leq C ( C_2 + C_5 + C_6),
$$
and so $\{ u_n \}_{n \geq 2}$ is bounded in $W^{2,1}_r (Q_T)$ for any $r \in (1,p)$. Since $p$ can be chosen as close to $2$ as we like, we choose $r \in (\tfrac{3}{2},2)$ and apply the Sobolev embedding to conclude that in fact
\begin{align}\label{uuniformbound}
\norm{u_n}_{C^{\sigma, \sigma/2} (\overline{Q}_T)} \leq \tilde C \norm{u_n}_{W^{2,1}_r (Q_T)} \leq \tilde C C (C_2 + C_5 + C_6 ),
\end{align}
for any $\sigma \in (0,\tfrac{1}{2})$, for some $\tilde C > 0$. In particular, $u_n$ is uniformly bounded in $\overline{Q}_T$, independent of $n$.

Now we are ready to obtain bounds on the time derivative $(u_n)_t$. While the previous step gives $L^p$-bounds on the time derivative for $p \in (1,2)$ only, with a bit of extra work we can show that it also holds for $p=2$. These estimates follow from standard arguments used in the development of the $L^2$-theory of parabolic equations (see, e.g., \cite[Ch. 3.3]{wu2006elliptic}), using all previous bounds. We show the key details only. 

First, note that bound \eqref{bnd3.2} in Theorem \ref{ODEbounds2} paired with the uniform boundedness of the iterates $\{ u_n \}_{n \geq 2}$ over $\overline{Q}_T$ obtained in \eqref{uuniformbound} implies that in fact $\{(k_n)_x\}_{n\geq2}$ is uniformly bounded in $L^{2,2} (Q_T)$. Multiplying the equation for $u_n$ by $(u_n)_t$ and integrating over $\Omega$ gives
\begin{align}
\int_\Omega \as{(u_n)_t}^2{\rm d}x &= \int_\Omega (u_n)_t \left( ( d (u_n)_x + \alpha u_n \overline{(k_n)}_x )_x + f(u_n) \right){\rm d}x
\end{align}
By the regularity of the iterates $(u_n,k_n)$ for fixed $n$, we may exchange the order of differentiation and integrate by parts to obtain
\begin{align}
\int_\Omega \as{(u_n)_t}^2{\rm d}x &= -\frac{d}{2} \int_\Omega \left( \as{(u_n)_x}^2 \right)_t {\rm d}x + \int_\Omega (u_n)_t f(u_n){\rm d}x\nonumber \\
&+ \alpha \int_\Omega (u_n)_t \left( (u_n)_x \overline{(k_n)}_x + u_n \overline{(k_n)}_{xx} \right){\rm d}x . 
\end{align}
Integrating from $0$ to $T$ and dropping the negative term, we are left with
\begin{align}
\iint_{Q_T} \as{(u_n)_t}^2 {\rm d}x{\rm d}t&\leq \frac{d}{2} \norm{(u_{0})_x}_{L^2 (\Omega)} ^2 + \iint_{Q_T} (u_n)_t f(u_n){\rm d}x\nonumber \\
&+ \alpha \iint_{Q_T} (u_n)_t \left( (u_n)_x \overline{(k_n)_x} + u_n \overline{(k_n)_{xx}} \right){\rm d}x . 
\end{align}
We then estimate crudely as follows: since $u_n$ and $\overline{(k_n)}_x$ are uniformly bounded in $\overline{Q}_T$ by \eqref{uuniformbound} and \eqref{Y1.2.33}, and since $(u_n)_x$ and $\overline{(k_n)}_{xx}$ are uniformly bounded in $L^{2,2}(Q_T)$ by \eqref{uxest} and preceding arguments, $f(u_n)$, $(u_n)_x \overline{(k_n)}_x$ and $u_n \overline{(k_n)}_{xx}$ are all uniformly bounded in $L^{2,2}(Q_T)$. Hence, a simple application of Cauchy's inequality with epsilon yields the existence of a constant $\tilde C^\prime >0$, independent of $n$, such that
\begin{align}
\frac{1}{2} \int_{Q_T} \as{(u_n)_t}^2{\rm d}x{\rm d}t &\leq  \frac{d}{2} \norm{(u_{0})_x}_{L^2 (\Omega)} ^2 + \tilde C ^\prime .
\end{align}
We now summarize the uniform estimates we have obtained and complete the limiting process.
\begin{align}
u_n &\in L^{\infty} (0,T; L^p (\Omega)) \cap C^{\sigma, \sigma/2} (\overline{Q}_T),\nonumber \\
(u_n)_x &\in L^{2} (0,T; L^{2} (\Omega)) , \quad (u_n)_t \in L^{2}( 0,T; L^{2} (\Omega)), \nonumber \\
k_n &\in L^{\infty} (0,T; L^{\infty} (\Omega)) \cap C^{\sigma, \sigma/2} ( \overline{Q}_T), \nonumber \\
(k_n)_x &\in L^{2} (0,T; L^{2} (\Omega)), \quad (k_n)_t \in L^{2} (0,T; L^{2} (\Omega)).
\end{align}
Hence, there exists a limit function $(u_\infty, k_\infty)$ so that for any $1 \leq p \leq \infty$ and any $0 < \sigma^\prime < \sigma < 1/2$, there holds
\begin{align}
u_n \to u_\infty,\ k_n \to k_\infty\quad \text{ strongly in }& L^{p,p} (Q_T) \cap C^{\sigma^\prime,\sigma^\prime/2} (\overline{Q}_T), \nonumber \\
(u_n)_x \to (u_\infty)_x\quad \text{ strongly in }& L^{2,2} (Q_T), \nonumber \\
(u_n)_t \to (u_\infty)_t\quad \text{ weakly in }& L^{2,2} (Q_T), \nonumber \\
(k_n)_t \to (k_\infty)_t,\ (k_n)_x \to (k_\infty)_x\quad \text{ weakly in }& L^{2,2} (Q_T).
\end{align}
It is not difficult to verify that $(u_\infty, k_\infty)$ is indeed a weak solution to the original problem \eqref{eq:general} in the sense of \eqref{weaksolnu}-\eqref{weaksolnk} and satisfies the initial data in the classical sense. Since  $u_n, k_n$ are nonnegative for all $n\geq 2$, we find that $0 \leq u_\infty, k_\infty$ in $\overline{Q}_T$. Furthermore, the solution $u_\infty$ is nontrivial since $f^\prime (0) > 0$, whence $k_\infty$ is also nontrivial. We now write $(u,k)$ for the solution obtained.

Uniqueness given initial data $(u_0,k_0)$ follows from standard arguments and using the fact that $u$ and $\overline{k}_x$ are uniformly bounded over $Q_T$. Indeed, if there were two solution pairs $(u,k)$ and $(\tilde u, \tilde k)$ satisfying the same initial data, an application of Cauchy's inequality with epsilon paired with the uniform boundedness of the solutions over $Q_T$, the smoothness of the functions $g_i(\cdot)$, $i=1,2$, $f(\cdot)$ (Lipschitz continuity is sufficient), and the linearity of the spatial convolution operation yields
$$
\frac{1}{2} \frac{{\rm d}}{{\rm d} t} \int_\Omega \left( (u - \tilde u)^2 + (k - \tilde k)^2 \right) {\rm d} x \leq C \int_\Omega \left( (u - \tilde u)^2 + (k - \tilde k)^2 \right) {\rm d} x,
$$
and so Gr{\"o}nwall's inequality implies that $\norm{(u-\tilde u) (\cdot,t)}_{L^2 (\Omega)} = \norm{(k-\tilde k) (\cdot,t)}_{L^2 (\Omega)} = 0$ for any $t \in (0,T)$, and uniqueness is proved.

Hence, for problem \eqref{1a}, there exists a unique, global weak solution in the sense of \eqref{weaksolnu}-\eqref{weaksolnk} so long as $g(u)$ satisfies the bound
$$
g(z) \leq M (\mu + \beta z) \quad \forall z \geq 0,
$$
for $\mu,\ \beta \geq 0$ fixed, for some $M>0$. For problem \eqref{1b}, $g(z) \leq \mu + \beta z + g(z)$ holds trivially, and no further condition on $g$ is required, concluding the proof.
\end{proof}

\section{Stability of spatially-constant steady states}\label{sec:stabilityanalysis}

\subsection{Spatially-constant steady states}

Under assumptions (H0)-(H2), system \eqref{1a} has two constant steady-states $(0,0)$ and $(1,\dfrac{\rho}{\mu+\beta}).$ 
For simplicity, denote 
\begin{align*}
%\begin{split}
h(u,k):= g(u) - (\mu + \beta u)k,\ U_0:=(0,0),\ U_*:{=}(1,\frac{\rho}{\mu+\beta}).
%\end{split}
\end{align*}
The  ODE kinetic system corresponding to \eqref{1a} is given by
\begin{equation}\label{1ODE}
\begin{cases}
u' = f(u), & \ t>0, \\
k' = g(u) - ( \mu + \beta u) k, & \ t> 0.
\end{cases}
\end{equation}
Then the Jacobian matrices $J_0$ at $U_0$ and $J_*$ at $U_*$ of \eqref{1ODE} are given by
\begin{align*}
J_0=\left(\begin{array}{cc}
f_{u0}&0 \\ 
0& -\mu
\end{array} \right),
\ J_*=\left(\begin{array}{cc}
f_{u*}&0 \\ 
h_{u*}& h_{k*}
\end{array} \right),
\end{align*}
where 
\begin{equation}\label{3}
\begin{split}
f_{u0}&:{=}f'(0)>0,\;\; h_{u*}:{=}h_u(U_*)=\frac{g'(1)(\mu+\beta)-\beta\rho}{\mu+\beta},\\
f_{u*}&:{=}f'(1)<0,\;\;  h_{k*}:{=}h_k(U_*)=-(\mu+\beta)<0.
\end{split}
\end{equation}

Let $Tr(J_*)$ denotes the trace of $J_*$, and let $Det(J_*)$ be the determinant of $J_*$. Then we have 
\begin{equation}\label{12}
%w_{u*}>0,\ \ 
Tr(J_*)=f_{u*}+h_{k*}<0,\ \ Det(J_*)=f_{u*}h_{k*}>0.
\end{equation}
Hence $U_*$ is a locally asymptotically stable steady state with respect to \eqref{1ODE}, and $U_0$ is linearly unstable with \eqref{1ODE} and also \eqref{1a}. 

\subsection{Linear stability analysis}
\label{sec:lsa}

In Section \ref{sec:spectral} we will use spectral analysis to determine rigorously the regions of stability for the constant steady state.  However, the formalism required for spectral analysis can obscure the central message.  Therefore it is valuable first to perform linear stability analysis using an ansatz. This gives a quick route to an answer that relies on said ansatz, which we then make rigorous via a more detailed spectral analysis.  

The ansatz we use is to assume that non-constant perturbations of the constant steady state have the following form at arbitrarily small times
\begin{align}
\label{eq:lsa_ansatz}
\tilde{u}=u_0 {\rm e}^{{\rm i}q_n x+\lambda t}, \quad \tilde{k}=k_0 {\rm e}^{{\rm i}q_n x+\lambda t}, \quad u \approx \tilde{u}+1, \quad k \approx \tilde{k} +\frac{\rho}{\mu+\beta},
\end{align}
where $u_0, k_0,\lambda \in {\mathbb R}$  are constants and $q_n=\sqrt{l_n}=n\pi/L$ for $n \in {\mathbb N}$.  These particular wavenumbers, $q_n$, are chosen as they satisfy the periodic boundary conditions. Then, neglecting nonlinear terms and applying Fourier theory, the PDEs in system \eqref{1a} become
\begin{align}
   \lambda \left(\begin{array}{c} \tilde{u} \\ \tilde{k}\end{array}\right) &= M_n\left(\begin{array}{c}\tilde{u} \\ \tilde{k}\end{array}\right),  
\end{align}
where 
\begin{align}
M_n&= \left(\begin{array}{cc} -l_n d+f'(1) & -l_n \alpha C_n(G) \\ 
   g'(1)-\frac{\beta\rho}{\mu+\beta} & -\mu-\beta\end{array}\right),
\end{align}
and $C_n(G)$ (the Fourier coefficient of $G$) is defined in \eqref{7}.

Stability requires that the trace of $M_n$ is negative and the determinant positive.  For the determinant to be positive, we require 
\begin{align}
\label{eq:lsa1}
\alpha C_n(G) [g'(1)(\mu+\beta)-\beta\rho]+(\mu+\beta)^2\left(d-\frac{f'(1)}{l_n}\right)>0.
\end{align}
%{\color{orange} [Di: $C_n(G)$ may change sign, so $\alpha>\alpha_n$ or $\alpha<\alpha_n$ can not be determined directly.]}
\eqref{eq:lsa1} holds when $\alpha C_n(G)$ is small or zero as $f'(1)<0$ by (H0), and \eqref{eq:lsa1} is true for all $n$ if $\alpha$ (positive or negative) is sufficiently close to $0$.
For the trace to be negative, we require 
\begin{align}
\label{eq:lsa2}
f'(1)< l_n d+\mu+\beta,
\end{align}
which is always true as the right-hand side is positive and $f'(1)<0$ by (H0). As long as Equations \eqref{eq:lsa1}-\eqref{eq:lsa2} are satisfied, system \eqref{1a} will be stable to perturbations of the exponential functional form given in Equation (\ref{eq:lsa_ansatz}) at wavenumber $q_n$. A similar process gives the analogous result for system \eqref{1b}, which we leave as an exercise for the reader.  In the next section, we generalise this result to arbitrary perturbations, for both systems \eqref{1a} and \eqref{1b} (see Theorems \ref{thm:2.8} and \ref{thm:2.9}).

\subsection{Spectral analysis}
\label{sec:spectral}

We now provide a detailed spectral analysis to confirm that the insights in Section \ref{sec:lsa} hold. Equation (3.6) is the eigenvalue problem to be examined here but requires further justification, which is done in Lemma 3.3 and 3.4. This ensures that the Fourier analysis utlized is robust. Of note is the symmetry of the kernel $G(\cdot)$ about the origin, which guarantees that the coefficients $C_n(G)$ are real-valued. If $G$ is not even, then $C_n(G)$ could be complex-valued, and may lead to a Hopf bifurcation. We do not explore this any further in the present work. Finally Theorem 3.6 shows other than the eigenvalues from (3.6), there is an essential spectral point. In this case, that essential spectral point is entirely negative and so will not affect stability. In general reaction-diffusion systems, the essential spectrum will not occur, but for such coupled PDE-ODE system, it may, and so we rule out the possibility. 

The linearized equation of system \eqref{1a}
at a constant steady state $U_*=(u_*,k_*)$ is given by
\begin{equation}\label{2.5}
	\begin{cases}
	\widetilde{u}_t=d\widetilde{u}_{xx}+f_{u*}\widetilde{u}+\alpha u_*(G\ast\widetilde{k})_{xx}, &x\in(-L,L),\ t>0,\\
	\widetilde{k}_t=h_{u*}\widetilde{u}+h_{k*}\widetilde{k}, &x\in(-L,L),\ t>0,\\
		\widetilde{u}(-L,t)=	\widetilde{u}(L,t)=0, & t>0,\\
		\widetilde{u}_x(-L,t)=	\widetilde{u}_x(L,t)=0, & t>0.
	\end{cases}
\end{equation}
Define the linearized operator $ \mathcal{L_*}(\alpha): X \rightarrow Y $ in \eqref{2.5} by 
\begin{equation}\label{20}
\mathcal{L_*}(\alpha)\left[\begin{array}{c}
\phi\\\psi 
\end{array} \right]=\left(\begin{array}{cc}
d \phi_{xx} +f_{u*}\phi + \alpha u_*(G\ast\psi)_{xx} \\ 
h_{u*}\phi + h_{k*}\psi
\end{array} \right).
\end{equation}
For further spectral analysis, we first recall the following definitions and give some lemmas.
\begin{definition}\cite[Definition 2.2.1]{MagalRuan2018}
Let $ A: \mathcal{D}(L)\subset X\rightarrow X $ be a linear operator on a $\mathbb{K}$-Banach space $X$ with $\mathbb{K}=\mathbb{R}$ or $\mathbb{C}$. The resolvent set $\rho(A)$ of $A$ is the set of all points $\lambda\in\mathbb{K}$ such that $(\lambda I-A)^{-1}$ is a bijection from $\mathcal{D}(A)$ into $X$ and the inverse $(\lambda I-A)^{-1}$, called the
resolvent of $A$, is a bounded linear operator from $X$ into itself.
\end{definition}

\begin{definition}\cite[Definition 4.2.1]{MagalRuan2018}
Let $ A: \mathcal{D}(L)\subset X\rightarrow X $ be a linear operator on a complex Banach space $X$. The spectrum of the operator $A$ is defined as the complement of the resolvent
set $\sigma(A)=\mathbb{C}\setminus \rho(A)$.
Consider the following three conditions:
\begin{itemize}[noitemsep,topsep=0pt]
	\item[(1)] $(\lambda I-A)^{-1}$ exists;
	\item[(2)] $(\lambda I-A)^{-1}$ is bounded;
	\item[(3)] the domain of $(\lambda I-A)^{-1}$ is dense in $X$.
\end{itemize}
The spectrum $\sigma(A)$ can be further decomposed into three disjoint subsets.
\begin{itemize}[noitemsep,topsep=0pt]
	\item[(a)] The point spectrum is the set
	\begin{equation*}
	\sigma_p(A): =\{\lambda\in\sigma(A): \mathcal{N}(\lambda I-A)\neq\{0\}\}.
	\end{equation*}
	Elements of the point spectrum $\sigma_p(A)$ are called eigenvalues. If $\lambda\in\sigma_p(A)$, elements $x\in\mathcal{N}(\lambda I-A)$ are called eigenvectors or eigenfunctions. The dimension of $\mathcal{N}(\lambda I-A)$ is the multiplicity of $\lambda$.
	\item[(b)] The continuous spectrum is the set
	\begin{equation*}
	\sigma_c(A): =\{\lambda\in\sigma(A): (1)\  \text{and}\ (3)\ \text{hold but}\ (2)\ \text{does not}\}.
	\end{equation*}
	\item[(c)] The residual spectrum is the set
	\begin{equation*}
	\sigma_r(A): =\{\lambda\in\sigma(A): (\lambda I-A)^{-1}\ \text{exists but}\ \overline{\mathcal{R}(\lambda I-A)}\neq X\}.
	\end{equation*}
\end{itemize}
Furthermore, we have the following spectrum decomposition:
\begin{equation*}
\sigma(A)=\sigma_p(A)\cup\sigma_c(A)\cup\sigma_r(A).
\end{equation*}
\end{definition}

\begin{lemma}\label{lem:2.3}
Assume $G$ satisfies (H0), and  $\phi\in H^2_{per}(-L,L)$. Then
\begin{equation}\label{8}
(G\ast\phi)_{xx}=G\ast(\phi_{xx}).
\end{equation}
\end{lemma}
\begin{proof}
By using the symmetry of $G$, the periodicity of $G$ and $\phi$, we get
\begin{align*}
	(G\ast\phi)_{xx}=&\dfrac{1}{2L}\int_{-L}^{L}G_{xx}(x-y)\phi(y)dy=\dfrac{1}{2L}\int_{-L}^{L}G_{yy}(x-y)\phi(y)dy\\=&\dfrac{1}{2L}(G_y(x+L)\phi(-L)-G_y(x-L)\phi(L)-\int_{-L}^{L}G_{y}(x-y)\phi'(y)dy)\\=&\dfrac{1}{2L}(G_y(x+L)\phi(-L)-G_y(x-L)\phi(L)-G(x-L)\phi'(L)
 +G(x+L)\phi'(-L)+\int_{-L}^{L}G(x-y)\phi''(y)dy)\\=&\dfrac{1}{2L}(G'(x+L)\phi(-L)-G'(x-L)\phi(L)-G(x-L)\phi'(L)
 +G(x+L)\phi'(-L))+G\ast(\phi_{xx})\\=&\ G\ast(\phi_{xx}).
	\end{align*}
\end{proof}
\begin{lemma}\label{lem:2.4}
	Assume $G$ satisfies (H0). Then 
	\begin{equation}\label{9}
	G\ast\phi_n=C_n(G)\phi_n,\ \ 
 G\ast\phi_{-n}=C_n(G)\phi_{-n}
	\end{equation}
where $C_n(G)$ is defined in \eqref{7} and $\phi_n, \phi_{-n}$ are defined in \eqref{6}.
\end{lemma}
\begin{proof}
	For $x\in[-L,L]$, 
	\begin{align*}
	G\ast\phi_n(x)=&\phi_n\ast G (x)=\frac{1}{2L}\int_{-L}^{L}\phi_n(x-y)G(y)dy=\frac{1}{2L}\int_{-L}^{L}e^{\frac{in\pi }{L}(x-y)}G(y)dy\\=&\frac{1}{2L}\int_{-L}^{L}e^{-\frac{in\pi }{L}y}G(y)dye^{\frac{in\pi }{L}x}=\frac{1}{2L}\int_{-L}^{L}\cos\left(\frac{n\pi }{L}y\right)G(y)dy\\=&\ C_n(G)\phi_n(x).
	\end{align*}
	Note that $C_n\in \mathbb{R}$ as $G$ is an even function from (H0). Thus the eigenspace corresponding to the eigenvalue $\lambda=C_n(G)$ is 
    \begin{align*}
        V_n := \spn\left\{\cos\left(\dfrac{n\pi}{L}x\right),\  \sin\left(\dfrac{n\pi }{L}x\right)\right\}.
    \end{align*}
\end{proof}

Following a similar approach to the proof of \cite[Proposition 2.1]{DucrotFuMagal2018}, we obtain the following lemma.
\begin{lemma}\label{lem:2.5}
	The spectrum of the linear operator $\mathcal{A}: D(\mathcal{A})\subset L^2_{per}(-L,L)\rightarrow L^2_{per}(-L,L)$ defined as
	\begin{equation*}
	\begin{cases}
	D(\mathcal{A})=H^2_{per}(-L,L),\\
		\mathcal{A}\phi=a\phi''+b(G\ast\phi''),
		\end{cases}
	\end{equation*}
	is
	\begin{equation*}
		\sigma(\mathcal{A})=\{\mu_n=\mu_{-n}: =-l_n(a+bC_n(G)),\ n\in\mathbb{N}_0\},
	\end{equation*}
 	and the corresponding eigenfunctions are $\phi_n(x)$ and $\phi_{-n}(x)$ for $n\in\mathbb{N}_0$, where $a, b$ are constants, and $l_n$, $\phi_n$, $\phi_{-n}$ are defined in \eqref{6}.

\end{lemma}

Now we can determine the spectral set of the linearized operator $\mathcal{L_*}(\alpha)$.
\begin{theorem}\label{thm:2.5}
Assume that assumptions (H0)-(H2) are satisfied. Let $l_n$ and $\phi_n$ be the eigenvalues and eigenfunctions of problem \eqref{2}. Then 
	\begin{equation}
	\sigma(\mathcal{L_*}(\alpha))=\sigma_p(\mathcal{L_*}(\alpha))=\{\lambda_n^{\pm}\}_{n\in\mathbb{Z}}\bigcup\{h_{k*}\},
	\end{equation}
	where 
	\begin{equation}\label{14}
	\begin{split}
	&\lambda_n^{\pm}=\lambda_{-n}^{\pm}=\dfrac{B_n\pm\sqrt{B_n^2-4C_n}}{2},\\
	& B_n=Tr(J_*)-dl_n, \;\; C_n=Det(J_*)+(\alpha u_* h_{u*}C_n(G)-dh_{k*})l_n,
	\end{split}
	\end{equation}
	for $n\in\mathbb{N}_0$, and
	\begin{equation}\label{16}
        \begin{aligned}
	\left(\phi_{n,\pm}, \psi_{n,\pm}\right)=\left(\phi_n(x),-\dfrac{h_{u*}}{h_{k*}-\lambda_n^\pm}\phi_n(x)\right),\\
 \left(\phi_{-n,\pm}, \psi_{-n,\pm}\right)=\left(\phi_{-n}(x),-\dfrac{h_{u*}}{h_{k*}-\lambda_n^\pm}\phi_{-n}(x)\right)
        \end{aligned}
	\end{equation}
 are the eigenfunctions corresponding to $\lambda_n^\pm$ and $\lambda_{-n}^\pm$, where $\phi_n$, $\phi_{-n}$, $h_{u*}$, $h_{k*}$, $Tr(J_*)$, $Det(J_*)$ and $C_n(G)$ are defined in \eqref{6}, \eqref{3}, \eqref{12} and \eqref{7}, respectively. Furthermore, $\lambda_n^{\pm}$ and $\lambda_{-n}^\pm$ are  eigenvalues of $\mathcal{L_*}(\alpha)$ of finite multiplicity, and $h_{k*}$ is an eigenvalue of infinite multiplicity.
\end{theorem}

\begin{proof}
	For $\lambda\in\mathbb{C}$ and $(\xi,\eta)\in Y$, we consider the resolvent equation of $\mathcal{L_*}(\alpha)$, which is
	\begin{equation}\label{4}
	\begin{cases}
	d\phi_{xx}+f_{u*}\phi+ \alpha u_*(G\ast\psi)_{xx}=\lambda\phi+\xi,  \\ 
	h_{u*}\phi+h_{k*}\psi=\lambda\psi+\eta,\\
	\phi(-L)=\phi(L), \; \phi'(-L)=\phi'(L). 
	\end{cases}
	\end{equation}
If $\lambda\neq h_{k*}$, from the second equation of \eqref{4}, we have 
\begin{equation}\label{5}
\psi=\dfrac{\eta-h_{u*}\phi}{h_{k*}-\lambda},
\end{equation}
Substituting \eqref{5} into the first equation of \eqref{4} and combining with \eqref{8} and  \eqref{9}, we get
\begin{equation}\label{10}
(f_{u*}-\lambda)\phi+(d\phi_{xx}-\alpha u_*\dfrac{h_{u*}}{h_{k*}-\lambda}(G\ast\phi''))=\xi.
\end{equation}
Equation \eqref{10} has non-zero solutions if and only if
\begin{equation}\label{11}
(f_{u*}-\lambda)\notin\sigma(\mathcal{A})
\end{equation}
holds, where $\mathcal{A}$ is the operator defined in Lemma \ref{lem:2.5} with $a=d$ and $b=-\alpha u_*\dfrac{h_{u*}}{h_{k*}-\lambda}$.
Then from Lemma \ref{lem:2.5}, \eqref{11} is equivalent to 
\begin{equation}\label{2.18}
	\dfrac{(f_{u*}-\lambda)(h_{k*}-\lambda)}{d(h_{k*}-\lambda)-\alpha u_* h_{u*}C_n(G)}\notin\{l_n\}_{n\in\mathbb{N}_0}.
\end{equation}
 It follows that $\mathcal{L_*}(\alpha)-\lambda I$ has a bounded inverse $(\mathcal{L_*}(\alpha)-\lambda I)^{-1}$ when \eqref{2.18} is satisfied. Otherwise, $\lambda$ satisfies the following characteristic equation:
\begin{equation}\label{13}
\lambda^2-(Tr(J_*)-dl_n)\lambda+Det(J_*)+(\alpha u_* h_{u*}C_n(G)-dh_{k*})l_n=0. 
\end{equation}
Therefore, $\lambda_{\pm n}^\pm$ in \eqref{14} are the roots of \eqref{13} with  $\mathcal{R}e\lambda_{\pm n}^-\leq\mathcal{R}e\lambda_{\pm n}^+$, and \eqref{16} are eigenfunctions corresponding to $\lambda_{\pm n}^\pm$.

If $\lambda=h_{k*}$, we consider \begin{equation}\label{17}
(\mathcal{L_*}(\alpha)-h_{k*}I)\left(\begin{array}{c}
\phi\\\psi
\end{array} \right)=\left(\begin{array}{c}
0\\0
\end{array} \right),
\end{equation} 
that is
\begin{equation}
\begin{cases}
d\phi_{xx}+f_{u*}\phi+ \alpha u_*(G\ast\psi)_{xx}=h_{k*}\phi,  \\
h_{u*}\phi=0,\\
\phi(-L)=\phi(L), \phi'(-L)=\phi'(L). \\
\end{cases}
\end{equation}
Clearly, $\phi=0$ and $(G\ast\psi)_{xx}=0$, which imply that there exist non-zero solutions to \eqref{17}. Then $h_{k*}$ is also an eigenvalue of $\mathcal{L_*}(\alpha)$, and $\dimm ker(\mathcal{L_*}(\alpha)-h_{k*}I)=\infty$. 
\end{proof}

Note that $Tr(J_*)<0$, $Det(J_*)>0$ from the stability of $U_*$, and $l_n\geq0$, which yields $\mathcal{R}e\lambda_{\pm n}^-<0$. In other words, eigenvalues $w_{k*}$ and $\{\lambda_{\pm n}^-\}_{n\in \mathbb{N}}$ of $\mathcal{L_*}$ all have negative real parts, and  for $n=0$, we also have
\begin{equation*}
\lambda_0=\dfrac{Tr(J_*)+\sqrt{Tr(J_*)^2-4Det(J_*)}}{2}<0.
\end{equation*}
On the other hand, the sign of $\lambda_{\pm n}^+$ depends on the magnitude of $\alpha$. Recall $\alpha_n$ defined in \eqref{18},
 we immediately have the following proposition.

\begin{proposition}\label{pro:2.7}
 Assume that assumptions (H0)-(H2) are satisfied, and $n\in {\mathbb N}$ such that $C_n(G)\ne 0$. Let $\mathcal{L_*}(\alpha)$ and $\alpha_{n}$ be defined in \eqref{20} and \eqref{18}, respectively. Then $0$ is an eigenvalue  of $\mathcal{L_*}(\alpha)$ when $\alpha=\alpha_{n}$ with multiplicity of two for $n\in \mathbb{N}$, while other eigenvalues have non-zero real parts. Furthermore, 
		\begin{equation}\label{22}
		\mathcal{N}(\mathcal{L}_*(\alpha_n))=\spn\left\{\left(1,-\dfrac{h_{u*}}{h_{k*}}\right)\phi_n(x),\ 
  \left(1,-\dfrac{h_{u*}}{h_{k*}}\right)\phi_{-n}(x)\right\}.
  %\left(1,-\dfrac{w_{u*}}{w_{k*}}\right)\overline{\phi_n}(x)\right\}.
		\end{equation}
\end{proposition}

\begin{proof} From Lemma \ref{lem:2.4}, $C_n(G)$ is real-valued. Substituting \eqref{18} into \eqref{14} yields  $\lambda_{\pm n}^+=0$ for $n\in \mathbb{N}$. The conclusion of eigenfunctions follow from Lemma \ref{lem:2.5}.
\end{proof}

We can now prove Theorems \ref{thm:2.8} and \ref{thm:2.9}. 

\begin{proof}[Proof of Theorem \ref{thm:2.8}] From the assumptions, $h_{k*}<0$. We also know that  $\mathcal{R}e\lambda_{\pm n}^-<0$. Finally for $\alpha_l<\alpha<\alpha_r$, all $\lambda_{\pm n}^+$ are negative. Hence $U_*$ is 
 is locally asymptotically stable with respect to \eqref{1a}. And when $\alpha<\alpha_l$ or $\alpha>\alpha_r$, at least one of $\lambda_{\pm n}^+$ is negative, thus $U_*$ is unstable. 
\end{proof}

The proof of Theorem \ref{thm:2.9} is similar by
repeating the same analysis. Equation \eqref{1b} has two constant steady states $\widehat{U}_0=(0,0)$ and $\widehat{U}_*=(\widehat{u}_*,\widehat{k}_*)=(1,\frac{\rho \kappa}{\mu+\beta+\rho})$. Let $\widehat{h}(u,k):{=}g(u)(\kappa-k) - (\mu + \beta u)k$, we have
\begin{equation}\label{3.30}
    \widehat{h}_{u_*}=\frac{\kappa[g'(1)(\mu+\beta)-\beta\rho]}{\mu+\beta+\rho},\ \widehat{h}_{k_*}=-(\mu+\beta+\rho). 
\end{equation}
Other parts are similar to the ones for the proof of Theorem \ref{thm:2.8}.

\section{Bifurcation analysis}\label{sec:bifurcationanalysis}

In this section, we prove the existence of non-constant steady state solutions of \eqref{1a} through the Bifurcation from
Simple Eigenvalue Theorem \cite{Rabinowitz1971}. From \eqref{22}, the multiplicity of zero eigenvalue is two, so we will restrict the solutions to be even functions only to apply the bifurcation theorem.

For completeness, we first recall the abstract bifurcation theorem, the definition of $K$-simple eigenvalue and a perturbation result.
Consider  an abstract equation $F(\lambda,u)=0$, where $F:\mathbb{R}\times X\rightarrow Y$ is a nonlinear differentiable mapping, and $X, Y$ are Banach spaces. 
Crandall and Rabinowitz \cite{Rabinowitz1971} obtained the following classical Bifurcation from Simple Eigenvalue Theorem. 

\begin{theorem}\cite[Theorem 1.17]{Rabinowitz1971}\label{thm:2}
Suppose that $\lambda_0\in\mathbb{R}$ and $F:\mathbb{R}\times X\rightarrow Y$ is a twice continuously differentiable mapping and that 
\begin{enumerate}
		\item[(\romannumeral1)] $F(\lambda,0)=0$ for $\lambda\in\mathbb{R}$,
		\item[(\romannumeral2)] $\dimm(\mathcal{N}(F_u(\lambda_0,0)))=\codim(\mathcal{R}(F_u(\lambda_0,0)))=1$,
		\item[(\romannumeral3)] $F_{\lambda u}(\lambda_0,0)[\phi_0]
		\notin\mathcal{R}(F_u(\lambda_0,0))$ where $\mathcal{N}(F_{u}(\lambda_0,0))=\spn\{\phi_0\}\in X$.
\end{enumerate}
Let $Z$ be any complement of $\spn\{\phi_0\}$ in $X$, then there exist an open interval $\widehat{I}$ containing $0$ and continuous functions $\lambda:\widehat{I}\rightarrow\mathbb{R},$ $z:\widehat{I}\rightarrow Z,$ such that $\lambda(0)=\lambda_0,$ $z(0)=0,$ and  $u(s)=s\phi_0+sz(s)$ satisfies  $F(\lambda(s),u(s))=0.$ Moreover, $F^{-1}(\{0\})$ near $(\lambda_0,0)$ consists precisely of the curves $u=0$ and the curves  $\{(\lambda(s),u(s)):s\in \widehat{I}\}.$
\end{theorem}

\begin{definition}\cite[Definition 1.2]{Crandall1973}
Let $B(X,Y)$ denote the set of bounded linear maps of $X$ into $Y$, and let $T,K\in B(X,Y)$. Then $\mu\in\mathbb{R}$ is a $K$-simple eigenvalue of $T$ if 
\begin{equation*}
	\dimm\mathcal{N}(T-\mu K)=\codim\mathcal{R}(T-\mu K)=1,
	\end{equation*}
	and if $\mathcal{N}(T-\mu K)=\spn\{\phi_0\}$, $Kx_0\notin \mathcal{R}(T-\mu K)$.
\end{definition}

\begin{theorem}\cite[Theorem 1.16]{Crandall1973} \label{thm:2.3}
Let $\{(\lambda(s),u(s)):s\in \widehat{I}\}$ be the curve of nontrivial solutions in Theorem \ref{thm:2}. Then there exist continuously differentiable functions $r:(\lambda_0-\varepsilon,\lambda_0+\varepsilon)\rightarrow\mathbb{R}$, $z:(\lambda_0-\varepsilon,\lambda_0+\varepsilon)\rightarrow X$, $\mu:(-\delta,\delta)\rightarrow\mathbb{R}$, $w:(-\delta,\delta)\rightarrow X$, such that 
	\begin{equation}\label{2.10}
	\begin{split}
	F_u(\lambda,0)z(\lambda)=r(\lambda)Kz(\lambda),\ \ \lambda\in(\lambda_0-\varepsilon,\lambda_0+\varepsilon),\\
	F_u(\lambda(s),u(s,\cdot))w(s)=\mu(s)Kw(s),\ \ s\in(-\delta,\delta),
	\end{split}
	\end{equation}
	where $r(\lambda_0)=\mu(0)=0$, $z(\lambda_0)=w(0)=(\phi_0)$, $K:X\rightarrow Y$ is the inclusion map with $K(u)=u.$
	Moreover, near $s=0$ the functions
	$\mu(s)$ and $-s\lambda'(s)r'(\lambda_0)$ have the same zeroes and, whenever $\mu(s)\neq0$ the same sign
	and satisfy
	\begin{equation}\label{2.12}
	\lim\limits_{s\rightarrow0}\dfrac{-s\lambda'(s)r'(\lambda_0)}{\mu(s)}=1.
	\end{equation}
\end{theorem}

To apply the bifurcation theorems, we define a nonlinear mapping $F: \mathbb{R}\times X\rightarrow Y$ by
\begin{align*}
F(\alpha,U)=\left(\begin{array}{c}
d  u_{xx} + \alpha ( u (G\ast k)_{x})_{x} + f(u)
\\ 
g(u)-(\mu+\beta u)k
\end{array}\right),
\end{align*}
where $U=(u,k)$, then the Frech\'et derivative of $F$ at $(\alpha,U)=(\alpha_n,U_*)$ is
\begin{equation}\label{19}
\partial_U F(\alpha_n,U_*)\left(\begin{array}{c}
\phi\\\psi 
\end{array} \right)=\left(\begin{array}{cc}
d \phi_{xx} +f_{u*}\phi + \alpha_n u_*(G\ast\psi)_{xx}  \\ 
h_{u*}\phi + h_{k*}\psi
\end{array} \right) =\mathcal{L}_{*}(\alpha_n)\left(\begin{array}{c}
\phi\\\psi 
\end{array} \right).
\end{equation}
For simplicity, we denote $\mathcal{L}_{*}(\alpha_n)$ as $\mathcal{L}_{n}$ in the following.
Let
\begin{align*}
%\begin{split}
X^s=\{h\in X : h(-x)=h(x), x\in (-L,L)\},\\
Y^s=\{h\in Y : h(-x)=h(x), x\in (-L,L)\}.
%\end{split}
\end{align*}
We consider the restriction of $F: {\mathbb R}\times X^s \to Y^s$, and the restriction of $\mathcal{L}_{n}: X^s\to Y^s$.
Denote $\mathcal{N}^s(\mathcal{L}_{n})$ to be the kernel space of the operator $\mathcal{L}_{n}$ in $X^s$, and $\mathcal{N}^s(\mathcal{L}_{n}^*)$ to be  the kernel space of the adjoint operator $\mathcal{L}_{n}$ in $X^s$.
Then $X^s$ and $Y^s$ have the following decompositions:
\begin{align*}
X^s=\mathcal{N}^s(\mathcal{L}_{n})\oplus X^s_1,\ Y^s=\mathcal{N}^s(\mathcal{L}_{n})\oplus Y^s_1,
\end{align*}
where
%\begin{equation}
\begin{align*}
&\mathcal{N}^s(\mathcal{L}_{n})=\spn\left\{\left(1,-\dfrac{h_{u*}}{h_{k*}}\right)\cos\left(\dfrac{n\pi}{L}x\right)\right\},\\ &\mathcal{N}^s(\mathcal{L}_{n}^*)=\spn\left\{(1,r_n)\cos\left(\dfrac{n\pi}{L}x\right)\right\} \ \text{with}\  r_n=\dfrac{dl_n-f_{u*}}{h_{u*}}, \\
&X^s_1=\left\{ (h_1,h_2)\in X^s: \int_{-L}^{L}\left(h_1-\dfrac{h_{u*}}{h_{k*}}h_2\right)\cos\left(\dfrac{n\pi}{L}\right)dx=0 \right\},\\
&Y^s_1=\mathcal{R}^s(\mathcal{L}_{n})=\left\{(h_1,h_2)\in Y^s : \int_{-L}^{L}(h_1+r_n h_2)\cos\left(\dfrac{n\pi}{L}x\right)dx=0\right\}.
\end{align*}
%\end{equation}
Hence, $\dimm(\mathcal{N}^s(\mathcal{L}_{n}))=\codim(\mathcal{R}^s(\mathcal{L}_{n}))=1.$ We also have
\begin{equation*}
\partial _{\alpha U}F(\alpha_{n},U_*)\left(\begin{array}{c}
\phi\\\psi 
\end{array} \right)=\left(\begin{array}{cc}
(G\ast\psi)_{xx} \\ 
0
\end{array} \right),
\end{equation*}
thus
\begin{align*}
%\begin{split}
    \partial _{\alpha U}F(\alpha _{n},U_*)\left[\left(1,-\frac{h_{u*}}{h_{k*}}\right)\cos\left(\dfrac{n\pi}{L}x\right)\right]^T
    =\left(-\frac{h_{u*}}{h_{k*}}\frac{\partial^2}{\partial x^2}\left(G\ast\cos\left(\dfrac{n\pi}{L}x\right)\right),0\right)^T\notin\mathcal{R}^s\left(\mathcal{L}_{n}\right)
%\end{split}
\end{align*}
as
\begin{equation*}
\frac{h_{u*}}{h_{k*}}l_{n}C_n(G)\int_{-L}^{L}\cos^2\left(\dfrac{n\pi}{L}x\right)dx\neq0,
\end{equation*}
as long as $C_n(G)\ne 0$.
Now by applying Theorem \ref{thm:2}, we obtain the existence of non-constant steady-state solutions of \eqref{1a} (Theorem \ref{thm:3.4}) and \eqref{1b}  (Theorem \ref{thm:3.5}).

Near a bifurcation point $\alpha=\alpha_n$, it follows from \cite{Shi1999} that the sign of $\alpha_n'(0)$ or the one of $\alpha_n''(0)$ when $\alpha_n'(0)=0$ determine the bifurcation direction. If $\alpha_n'(0)\ne 0$, then a transcritical bifurcation occur and a non-trivial solution exists when $\alpha(\ne \alpha_n)$ is close to the bifurcation point $\alpha_n$. If $\alpha_n'(0)=0$ and $\alpha_n''(0)\ne 0$, then a pitchfork bifurcation occurs at $\alpha=\alpha_n$. The pitchfork bifurcation is forward if $\alpha_n''(0)>0$ and there are two (zero) non-trivial solutions for $\alpha>\alpha_n$ ($\alpha<\alpha_n$), and it is backward if $\alpha_n''(0)<0$.
Since 
\begin{equation}\label{3.8}
\begin{aligned}
    &\langle \zeta,F_{UU}(\alpha_{n},U_*)\left[\left(1,-\dfrac{h_{u*}}{h_{k*}}\right)\cos\left(\dfrac{n\pi}{L}x\right)\right]^2\rangle\\=
    &\int_{-L}^{L}\Biggl[\left(f_{uu*}+r_n\left(h_{uu*}-2h_{uk*}\dfrac{h_{u*}}{h_{k*}}\right)\right)\dfrac{1+\cos\left(\dfrac{2n\pi}{L}x\right)}{2}\nonumber +2\alpha_n l_n C_n(G) \dfrac{h_{u*}}{h_{k*}}\cos\left(\dfrac{2n\pi}{L}x\right)\Biggr]\cos\left(\dfrac{n\pi}{L}x\right)dx
    =0,
   \end{aligned}
\end{equation}
where $f_{uu*}\stackrel{\bigtriangleup}{=}f_{uu}(U_*),\  h_{uu*}\stackrel{\bigtriangleup}{=}h_{uu}(U_*),\ h_{uk*}\stackrel{\bigtriangleup}{=}h_{uk}(U_*)=-\beta$, and $\zeta\in (Y^s)^*$ satisfying $N(\zeta)=\mathcal{R}(\mathcal{L}_{n})$. Then
\begin{equation}\label{46}
\begin{split}
\alpha_n'(0)=&-\dfrac{\left\langle \zeta,F_{UU}(\alpha_{n},U_*)\left[\left(1,-\dfrac{h_{u*}}{h_{k*}}\right)\cos\left(\dfrac{n\pi}{L}x\right)\right]^2\right\rangle}{2\left\langle \zeta,F_{\alpha U}(\alpha_{n},U_*)\left[\left(1,-\dfrac{h_{u*}}{h_{k*}}\right)\cos\left(\dfrac{n\pi}{L}x\right)\right]\right\rangle}=0.
\end{split}
\end{equation}
We further calculate  $\alpha_n''(0)$ as in \cite{Shi1999},
\begin{equation}\label{3.10}
\begin{aligned}
\alpha_n''(0)=-\dfrac{\left\langle \zeta,F_{UUU}\left(\alpha_{n},U_*\right)\left[\left(1,-\dfrac{h_{u*}}{h_{k*}}\right)\cos\left(\dfrac{n\pi}{L}x\right)\right]^3\right\rangle}{3\left\langle \zeta,F_{\alpha U}\left(\alpha_{n},U_*\right)\left[\left(1,-\dfrac{h_{u*}}{h_{k*}}\right)\cos\left(\dfrac{n\pi}{L}x\right)\right]\right\rangle}
-\dfrac{\left\langle \zeta,F_{UU}(\alpha_{n},U_*)\left[\left(1,-\dfrac{h_{u*}}{h_{k*}}\right)\cos\left(\dfrac{n\pi}{L}x\right)\right]\left[\Theta\right]\right\rangle}{\left\langle \zeta,F_{\alpha U}\left(\alpha_{n},U_*\right)\left[\left(1,-\dfrac{h_{u*}}{h_{k*}}\right)\cos\left(\dfrac{n\pi}{L}x\right)\right]\right\rangle},
\end{aligned}
\end{equation}
where $\Theta=\left(\Theta_1,\Theta_2\right)$ is the unique solution of 
\begin{equation}\label{3.11}
    F_{UU}\left(\alpha_{n},U_*\right)\left[(1,-\dfrac{h_{u*}}{h_{k*}})\cos\left(\dfrac{n\pi}{L}x\right)\right]^2+F_{U}\left(\alpha_{n},U_*\right)\left[\Theta\right]=0.
\end{equation}
From \cite{ShiShiWang2021JMB} and \eqref{3.8}, we assume $\Theta=\left(\Theta_1,\Theta_2\right)$ has the following form
\begin{equation}\label{3.12}
    \Theta_1=\Theta_1^1+\Theta_1^2 \cos\left(\dfrac{2n\pi}{L}x\right),\ \  \Theta_2=\Theta_2^1+\Theta_2^2 \cos\left(\dfrac{2n\pi}{L}x\right).
\end{equation}
Combining \eqref{3.11} and \eqref{3.12}, after calculation, we have
\begin{align*}
    %\begin{aligned}
       \Theta_1^1&=-\dfrac{f_{uu*}}{2 f_{u*}}, \ \ \ \ \ \ \ \ \ \ \ \  \Theta_2^1=\dfrac{h_{u*}f_{uu*}-f_{u*}\left(h_{uu*}-2h_{uk*}\dfrac{h_{u*}}{h_{k*}}\right)}{2 f_{u*}h_{k*}},\\
       \Theta_1^2&=-\dfrac{h_{k*}\left(\dfrac{f_{uu*}}{2}+2\alpha_n l_n C_n(G) \dfrac{h_{u*}}{h_{k*}}\right)+2\alpha_n u_* l_n C_n(G)\left(h_{uu*}-2h_{uk*}\dfrac{h_{u*}}{h_{k*}}\right)}{h_{k*}(f_{u*}-4dl_n)+4\alpha_n u_* h_{u*} l_n C_n(G)},\\
       \Theta_2^2&=\dfrac{h_{u*}\left(\dfrac{f_{uu*}}{2}+2\alpha_n l_n C_n(G) \dfrac{h_{u*}}{h_{k*}}\right)-\dfrac{1}{2}\left(f_{u*}-4dl_n\right)\left(h_{uu*}-2h_{uk*}\dfrac{h_{u*}}{h_{k*}}\right)}{h_{k*}\left(f_{u*}-4dl_n\right)+4\alpha_n u_* h_{u*} l_n C_n(G)}.
   % \end{aligned}
\end{align*}
Hence, by \eqref{3.10}, $\alpha''(0)$ can be calculated as follows:
\begin{align*}
   % \begin{aligned}
       \alpha_n''(0)&=\dfrac{\dfrac{f_{uuu*}+r_n g_{uuu*}}{2}+f_{uu*}\left(\Theta_1^1+\dfrac{\Theta_1^2}{2}\right)+\alpha_n\dfrac{h_{u*}}{h_{k*}}l_n C_n(G)\left(\Theta_1^1-\dfrac{\Theta_1^2}{2}\right)}{\dfrac{h_{u*}}{h_{k*}} l_n C_n(G)}\\
       &\ \ +\dfrac{-\alpha_n l_n C_n(G)\Theta_2^2+r_n\left[\left(h_{uu*}-h_{uk*}\dfrac{h_{u*}}{h_{k*}}\right)\left(\Theta_1^1+\dfrac{\Theta_1^2}{2}\right)+h_{uk*}\left(\Theta_2^1+\dfrac{\Theta_2^2}{2}\right)\right]}{\dfrac{h_{u*}}{h_{k*}} l_n C_n(G)},
   % \end{aligned}
\end{align*}
where $f_{uuu*}\stackrel{\bigtriangleup}{=}f_{uuu}(U_*),\  g_{uuu*}\stackrel{\bigtriangleup}{=}g_{uuu}(U_*)$. If $\alpha_n''(0)>0$, then  a forward pitchfork bifurcation occurs; and  If $\alpha_n''(0)<0$, then  a backward pitchfork bifurcation occurs.

From Theorem \ref{thm:2.3}, we obtain the stability of the nonconstant steady-state solution of problem \eqref{1a} and \eqref{1b} obtained in Theorems \ref{thm:3.4} and \ref{thm:3.5}, respectively.

\begin{theorem}\label{thm:4.6}
	Suppose the conditions of Theorem \ref{thm:3.4} are satisfied, and let $(\alpha_n(s),U(s,\cdot))$ ($|s|<\delta$) be the non-constant steady state solutions bifurcating from the constant ones at $\alpha=\alpha_n$. Then a pitchfork bifurcation occurs at $\alpha=\alpha_n$ if $\alpha_n''(0)\ne 0$.
	\begin{enumerate}
		\item[(\romannumeral1)] At $\alpha=\alpha_r>0$, suppose that $\alpha_r=\alpha_N$ for $N\in {\mathbb N}$, then the pitchfork bifurcation is forward and the bifurcating solutions are locally asymptotically stable with respect to \eqref{1a} if $\alpha_N''(0)>0$, and it is backward and the bifurcating solutions are unstable if   $\alpha_N''(0)<0$. The other bifurcating solutions near $\alpha_n>0$ for $n\ne N$ are all unstable.
		\item[(\romannumeral2)] At $\alpha=\alpha_l<0$, suppose that $\alpha_l=\alpha_M$ for $M\in {\mathbb N}$, then the pitchfork bifurcation is backward and the bifurcating solutions are  locally asymptotically stable with respect to \eqref{1a} if $\alpha_M''(0)<0$, and it is foward and the bifurcating solutions are unstable if   $\alpha_M''(0)>0$. The other bifurcating solutions near $\alpha_n<0$ for $n\ne M$ are all unstable.
	\end{enumerate}
\end{theorem}

\begin{proof} Let $r(\alpha)$ be the eigenvalue of the corresponding linearization operator for the constant solution $U_*$ such that $r(\alpha_n)=0$. 
	From \eqref{2.10}, 
	\begin{equation}
F_{U}\left(\alpha,U_*\right)\left[\phi(\alpha),\psi(\alpha)\right]^T = r(\alpha) K \left[\phi(\alpha),\psi(\alpha)\right]^T,\ \alpha\in(\alpha_n-\epsilon, \alpha_n+\epsilon).
	\end{equation}
	Here $K: X\to Y$ is the inclusion map $i(U)=U$. From Theorem \ref{thm:2.5}, we thus get 
	\begin{equation*}
	   % r(\alpha)=\dfrac{(Tr(J_*)-dl_n)+\sqrt{(Tr(J_*)-dl_n)^2-4[Det(J_*)+(\alpha u_* w_{u*}C_n(G)-dw_{k*})l_n]}}{2},
      r(\alpha)=\dfrac{B_n+\sqrt{B_n^2-4C_n}}{2},
	\end{equation*}
 where $B_n$ and $C_n$ are defined in \eqref{14}. Moreover, from the definition of $B_n$ and $C_n$, we obtain
	\begin{equation}\label{413}
	    %r'(\alpha_{n})=\dfrac{2u_* w_{u*}C_n(G)l_n}{-(Tr(J_*)-dl_n)}.
     r'(\alpha_{n})=\dfrac{u_* h_{u*}C_n(G)l_n}{B_n}.
	\end{equation}
 In particular, we have $\sign(r'(\alpha_{n}))=\sign(\alpha_n)$ from \eqref{413} and \eqref{18}. From \eqref{46}, we have $\alpha_n'(0)=0$ for any $n\in {\mathbb N}$, and $\alpha''(0)$ can be calculated as in \eqref{3.10}. Thus a pitchfork bifurcation occurs at $\alpha=\alpha_n$ if $\alpha_n''(0)\ne 0$.

Suppose that $\alpha_r=\alpha_N>0$ for some $N\in {\mathbb N}$, then $r'(\alpha_N)>0$. If $\alpha_N''(0)>0$, then $\alpha_N'(s)>0$  for $s\in(0,\delta)$, and $\alpha_N'(s)<0$  for $s\in(-\delta,0)$. From Theorem \ref{thm:2.3}, $Sign(-s\alpha_N'(s)r'(\alpha_N))=Sign(\mu(s))$, where $\mu(s)$ is the eigenvalue of the corresponding linearization operator for the bifurcating solution at $\alpha=\alpha(s)$. Thus $\mu(s)<0$ for $0<|s|<\delta$. Since all other eigenvalues of linearized equation at $(\alpha_n(s),U(s,\cdot))$ are negative near $\alpha=\alpha_r$, then $(\alpha_n(s),U(s,\cdot))$ is locally asymptotically stable. Similarly if $\alpha_N''(0)<0$, then $\mu(s)>0$ for $0<|s|<\delta$ hence $(\alpha_n(s),U(s,\cdot))$ is unstable. The cases for bifurcation point $\alpha=\alpha_l<0$ or $\alpha\ne \alpha_l,\alpha_r$ can be proved in a similar way.
\end{proof}

\begin{theorem}\label{thm:4.7}
   Suppose the conditions of Theorem \ref{thm:3.5} are satisfied. Then the same results as in Theorem \ref{thm:4.6} hold for equation \eqref{1b}, with  $U_*$ and $\alpha_n$ replaced by  $\widehat {U}_*$ and $\widehat{\alpha}_n$.
\end{theorem}

When the bifurcating solutions near $\alpha=\alpha_l$ or $\alpha_r$ are locally asymptotically stable, one can observe a small amplitude non-constant steady state solution of \eqref{1a} with a prescribed wave pattern (corresponding to $N$ or $M$ in Theorem \ref{thm:4.6}). When the bifurcating solutions near $\alpha=\alpha_l$ or $\alpha_r$ are unstable, the corresponding bifurcating branch will bend back through a saddle-node bifurcation and likely to a large amplitude non-constant steady state solution of \eqref{1a}  with the same prescribed wave pattern.

\section{Analysis of the model with a top-hat detection function}\label{sec:top-hatanalysis}

In this section, we study some specific cases to demonstrate some of our analytical results corresponding to different growth functions $g(\cdot)$ which exhibit different rates of growth for large arguments. Depending on the functional form of memory uptake $g(\cdot)$, we establish a number of monotone/nonmonotone properties of the bifurcation values $alpha_n (R)$. In each subsection we plot the relevant bifurcation curves $\alpha_l (R)$, $\alpha_r (R)$, along with a depiction of the corresponding steady state profiles near and far away from these critical values. For the numerical simulations we use a pseudo-spectral method with a forward-Euler time-stepping scheme. Trajectories are run until subsequent time steps are within a tolerance of $10^{-6}$. Note that the spatial domain chosen is $(0,2 \pi)$, equivalent to choosing $\Omega = (-\pi, \pi)$.

To this end, let $L = \pi$ and $f(u) = u ( 1 - u )$. In all cases, we choose $d=\mu = \beta =1.0$ and $\rho=5$. We then choose the following three cases of $g(u)$ to analyze equation \eqref{1a} and equation \eqref{1b} with the top-hat detection function defined in \eqref{detectionkernelT}:\\

(\romannumeral1) $g(u) = \dfrac{2  \rho u^2}{1 + u^2}$;$\quad\quad$    (\romannumeral2) $g(u) = \dfrac{2  \rho u^2}{1 + u}$;$\quad\quad$   (\romannumeral3) $g(u) = \rho u^2$.\\

As previously noted in Remark \ref{rmk:1.2-1}, cases (\romannumeral1) and (\romannumeral2) have a global weak solution for either problem \eqref{1a} or \eqref{1b}. In case (\romannumeral3), a global weak solution is only guaranteed by Theorem \ref{thm:exist1a} for problem \eqref{1b}.
\subsection{\texorpdfstring{Case (\romannumeral1) in equation \eqref{1a}}{}}\label{sec:case1}
In this case, $g(u) = \dfrac{2  \rho u^2}{1 + u^2}$, obviously, $f$ and $g$ satisfy assumptions (H1)-(H2). One can calculate that $ ( u_* , k_* ) = \left( 1 , \dfrac{\rho}{\mu + \beta} \right) $ and 
\begin{equation}\label{4.1}
\begin{aligned}
    &f_{u*} = - 1, \ f_{uu*} = - 2, \ f_{uuu*} = 0, \ g_{u*} = \rho, \ h_{u*} = \dfrac{\rho \mu}{\mu + \beta},\\
    &h_{k*} = - ( \mu + \beta ), \
    h_{uu*} = g_{uu*} = - \rho,\ h_{uk*} = - \beta,\  h_{kk*} = 0,\\
    &h_{uuu*} = g_{uuu*} = h_{uuk*} = h_{ukk*} = h_{kkk*} = 0.
    \end{aligned}
\end{equation}
We consider the following model with the top-hat detection function:
\begin{equation}\label{4.0}
\begin{cases}
u_t = d u_{xx} + \alpha ( u \overline{k}_{x} )_{x} + u (1 - u), & x\in (-\pi,\pi), \ t>0, \\
k_t = \dfrac{2 \rho u^2}{1 + u^2} - ( \mu + \beta u) k, & x \in (-\pi,\pi), \ t> 0,
\end{cases}
\end{equation}
subject to periodic boundary conditions, where $G(x)$ is defined as in \eqref{detectionkernelT}  such that $0<R<\pi$. From \eqref{1.3}, we have
\begin{equation}
	\overline{k}(x)=\dfrac{1}{2 \pi}\int_{x-R}^{x+R} \frac{1}{2R}k(y) dy,\ -\pi\leqslant x\leqslant \pi,
\end{equation}
and for $G$ in \eqref{detectionkernelT}, we have
\begin{equation}\label{cng}
    C_n(G)=\frac{\sin (nR)}{2 \pi n R}.
\end{equation}
As in \eqref{20},  the linearized operator at $(\alpha,U_*)$ is
\begin{align}\label{4.2}
%\begin{split}
    &\mathcal{L}_*(\alpha)\left(\begin{array}{c}
\phi\\\psi 
\end{array} \right) = \partial_UF(\alpha,U_*)\left(\begin{array}{c}
\phi\\\psi 
\end{array} \right) 
= \left(\begin{array}{cc}
d \phi_{xx} -\phi + \frac{\alpha}{4RL}\left(\psi_x (x+R)-\psi_x (x-R)\right) \\ 
\frac{\rho\mu}{\mu+\beta} \phi -(\mu+\beta)\psi
\end{array} \right).
%\end{split}
\end{align}
%where $f_{u*},\ w_{u*},\ w_{k*}$ are defined in \eqref{4.1}.
From Theorem \ref{thm:2.5}, the spectrum of $\mathcal{L}_*(\alpha)$ is consisted of $h_{k*}=-(\mu+\beta)<0$ and eigenvalues $\lambda_n^{\pm}$ which satisfy the characteristic equation
\begin{equation}\label{4.10}
\lambda^2+(1+\mu+\beta+d n^2)\lambda+(\mu +\beta)+\left(\alpha \dfrac{\rho \mu}{\mu + \beta} \dfrac{\sin (nR)}{2 \pi n R} + d(\mu +\beta) \right)n^2=0,
\end{equation}
where $l_n$ is defined in \eqref{6}, $n\in\mathbb{Z}$. When $n=0,$ from \eqref{12}, all eigenvalues of \eqref{4.10} have negative real parts, hence the constant solution $U_*$ is locally asymptotically stable with respect to non-spatial dynamics.
Note that \eqref{4.10} is an even function of $n$, so we consider $n\in\mathbb{N}$ below. 

From \eqref{18},  \eqref{2.24} and \eqref{cng}, we obtain 
%\begin{equation}
\begin{align}
	\alpha_{n}^R &=\frac{Det(J_*)-dh_{k*}n^2}{u_* h_{u*}n^2\frac{\sin (nR)}{2 \pi n R}} = \dfrac{-2 \pi n R(\mu+\beta)^2}{\rho \mu \sin (nR)}\left(d + \dfrac{1}{n^2}\right),\nonumber\\
	\Sigma^+ &= \left\{n\in \mathbb{N}: nR\in \cup_{j=0}^{\infty} (2j\pi,(2j+1)\pi)\right\},\nonumber\\
 %\cup_{j=0}^{\infty} (j\pi,(j+1)\pi)\ \text{for}\ j\ \text{is even},\\
	\Sigma^- &= \left\{n\in \mathbb{N}: nR\in \cup_{j=0}^{\infty} ((2j+1)\pi,(2j+2)\pi)\right\},\label{4.12}\\
 %\cup_{j=0}^{\infty} (j\pi,(j+1)\pi)\ \text{for}\ j\ \text{is odd},\\
	\alpha_l &=-\frac{2\pi(\mu+\beta)^2}{\rho\mu}\min_{n\in\Sigma^+}\dfrac{nR}{\sin (nR)}\left(d + \dfrac{1}{n^2}\right),\nonumber\\
 %\max\limits_{n\in\mathbb{N}}\dfrac{-(\mu+\beta)^2}{\rho \mu \frac{\sin (nR)}{2 \pi n R}}(d + \dfrac{1}{n^2}),\ \text{for}\ n\in\Sigma^-,\\
	\alpha_r &=-\frac{2\pi(\mu+\beta)^2}{\rho\mu}\max_{n\in\Sigma^-}\dfrac{nR}{\sin (nR)}\left(d + \dfrac{1}{n^2}\right),\nonumber
	%\alpha_l &=\dfrac{-(\mu+\beta)^2}{\rho \mu \frac{1}{2 \pi }}(d + \dfrac{1}{n^2}),\ 
	%\alpha_r =\dfrac{-(\mu+\beta)^2}{\rho \mu \frac{1}{2\pi}\frac{\sin (z)}{z}}(d + \dfrac{1}{n^2}),\\
	\end{align}
%\end{equation} 
In Figure \ref{fig:1} we numerically compute the values $\as{\alpha_l}$ and $\alpha_r$ and plot them with respect to the perceptual radius $R$. Now we can apply Theorems \ref{thm:2.8} and \ref{thm:3.4} to \eqref{4.0} to have the following results.
\begin{theorem}\label{thm:4.1}
Let $\alpha_n^R, \Sigma^+, \Sigma^-, \alpha_l, \alpha_r$ be defined in \eqref{4.12}. Then the constant steady state solution $U_*=( 1 , \rho/(\mu + \beta))$ is locally asymptotically stable with respect to \eqref{4.0} when $\alpha_l<\alpha<\alpha_r$ and is unstable when $\alpha<\alpha_l$ or $\alpha>\alpha_r$. Moreover  non-constant steady state solutions of \eqref{4.0} bifurcate from the the branch of constant solutions $\Gamma_0=\{(\alpha,U_*):\alpha\in{\mathbb R}\}$  near $\alpha=\alpha_{n}^R$, and these solutions are on a curve $\Gamma_n=\{(\alpha_n(s),u_n(s,\cdot),k_n(s,\cdot)): |s|<\delta\}$ such that $\alpha_n(0)=\alpha_n^R$ and $\alpha_n'(0)=0$. Moreover, the following monotonicity properties hold.
\begin{enumerate}
    \item [(\romannumeral1)] Suppose $n \in \Sigma^+$ so that $\dfrac{\sin(nR)}{nR}>0$. Then $\alpha_n(R)<0$, and $\alpha_n(R)$ (in particular, $\alpha_l$) is monotonically increasing with respect to $\rho$, is monotonically decreasing with respect to $d$ and $\beta$, and is not monotone with respect to $\mu$.
    \item [(\romannumeral2)] Suppose $n \in \Sigma^-$ so that $\dfrac{\sin(nR)}{nR}<0$. Then $\alpha_n(R)>0$, and $\alpha_n(R)$ (in particular, $\alpha_r$) is monotonically increasing with respect to $d$ and $\beta$, is monotonically decreasing with respect to $\rho$, and is not monotone with respect to $\mu$.
\end{enumerate}
\end{theorem}
After calculation, 
\begin{equation*}
    \alpha_n''(0)=\frac{-2n^4-\frac{143}{24}n^2-\frac{5}{24}}{-\frac{5}{4} n^2 \frac{\sin (nR)}{2 \pi n R}}.
\end{equation*}
Therefore, at $\alpha=\alpha_r$, the pitchfork bifurcation is forward and the bifurcating solutions are locally asymptotically stable with respect to \eqref{4.0}, at $\alpha=\alpha_l$, the pitchfork bifurcation is backward and the bifurcating solutions are locally asymptotically stable with respect to \eqref{4.0}.

The monotonicity properties described above can be understood intuitively. In case (i), we consider cases of aggregation ($\alpha < 0$), and so a \textit{decreasing} behaviour requires higher rates of advection to destabalize the constant steady state, while an \textit{increasing} behaviour allows for destabalization of the constant steady state at lower advection rates. As is generally understood for diffusion-advection models, diffusion has a stabalizing effect, and higher rates of diffusion therefore require comparably high magnitudes of advection to destabalize the constant steady state. Similarly, an increased value of $\beta$, a `rate of safe return', also requires an increased magnitude of advection to destabalize the constant steady state. This suggests that the population cannot return to previously visited locations too quickly if patterns are to persist. In case (ii), we flip the sign of the advection rate and consider the segregataion case ($\alpha > 0$), in which case the direction of the monotonicities also switch, but the understanding of this behaviour is identical to case (i). Most interestingly, perhaps, is the nonmonotone behaviour with respect to the memory decay rate $\mu$. In fact, in this case $\alpha_l$ ($\alpha_r$) is concave down (up), and so there is a critical value $\mu^* > 0$ so that the rate of advection required to destabalize the constant steady state is minimal. This is in contrast to Theorem \ref{thm:monotone-2} in case (ii), where monotonicity with respect to $\beta$ is lost.

Figure \ref{fig:1} demonstrates the more complex relationship between the perceptual radius $R$ and the sizes of $\alpha_r(R)$ and $\alpha_l(R)$. Of note is the nonmonotone behaviour, particularly for smaller perceptual radii. This wavelike behaviour is most pronounced for $\alpha_r (R)$. It is also easy to see that $\lvert\alpha_l(R)\rvert<\alpha_r(R)$ when $0<R<\pi/2$ (indeed this holds for $R<2.2$ from Figure \ref{fig:1}). But when $R$ is larger than $2.2$, either $\lvert\alpha_l(R)\rvert$ or $\alpha_r(R)$ could be the larger one. This in general shows that the advection rate needed to destabilize the positive equilibrium is larger when the perceptual radius is larger. When the perceptual radius is less than half of domain size, the attractive advection rate needed to destabilize the positive equilibrium is larger than the repulsive one. In Figures \ref{fig:1b}-\ref{fig:1c}, we plot the solution profiles just before, just after, and far beyond the analytically calculated bifurcation points $\alpha_l^*$ and $\alpha_r^*$, respectively.

\begin{figure}[ht]
	\centering
	\includegraphics[width=0.75\textwidth]{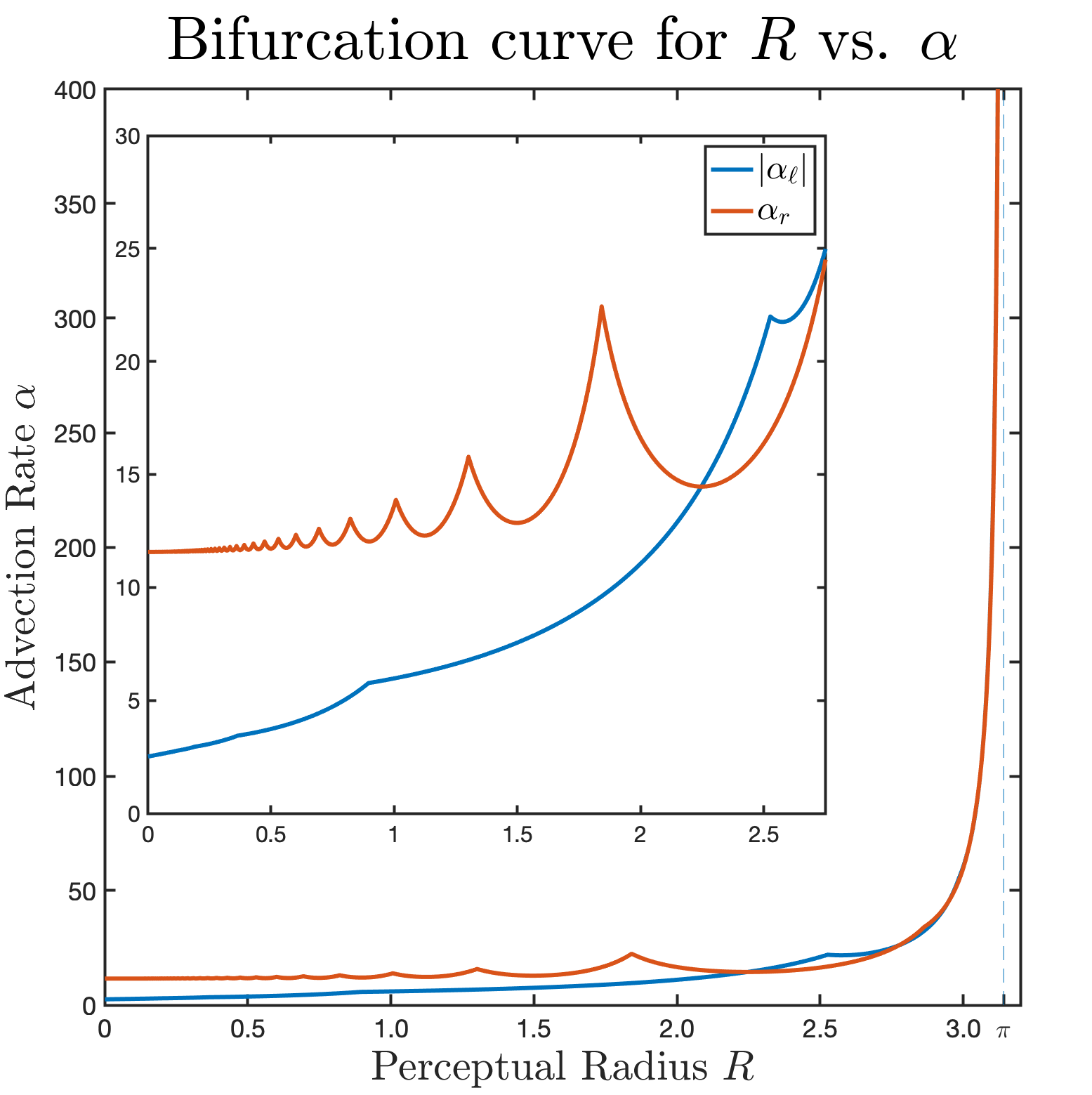}
	\caption{The bifurcation curves for Section \ref{sec:case1} for the perceptual radius $R$ versus the advection rate $\alpha$. In the subsequent Figures \ref{fig:1b}-\ref{fig:1c}, we fix $R=2.5$ and use the bifurcation curves to test values near and far from the critical values $\alpha_r^*$ and $\alpha_\ell ^*$ where the constant steady state is expected to be destabalized. 
	}
	\label{fig:1}
\end{figure}

\begin{figure}[ht]
	\centering
	\includegraphics[width=1.0\textwidth]{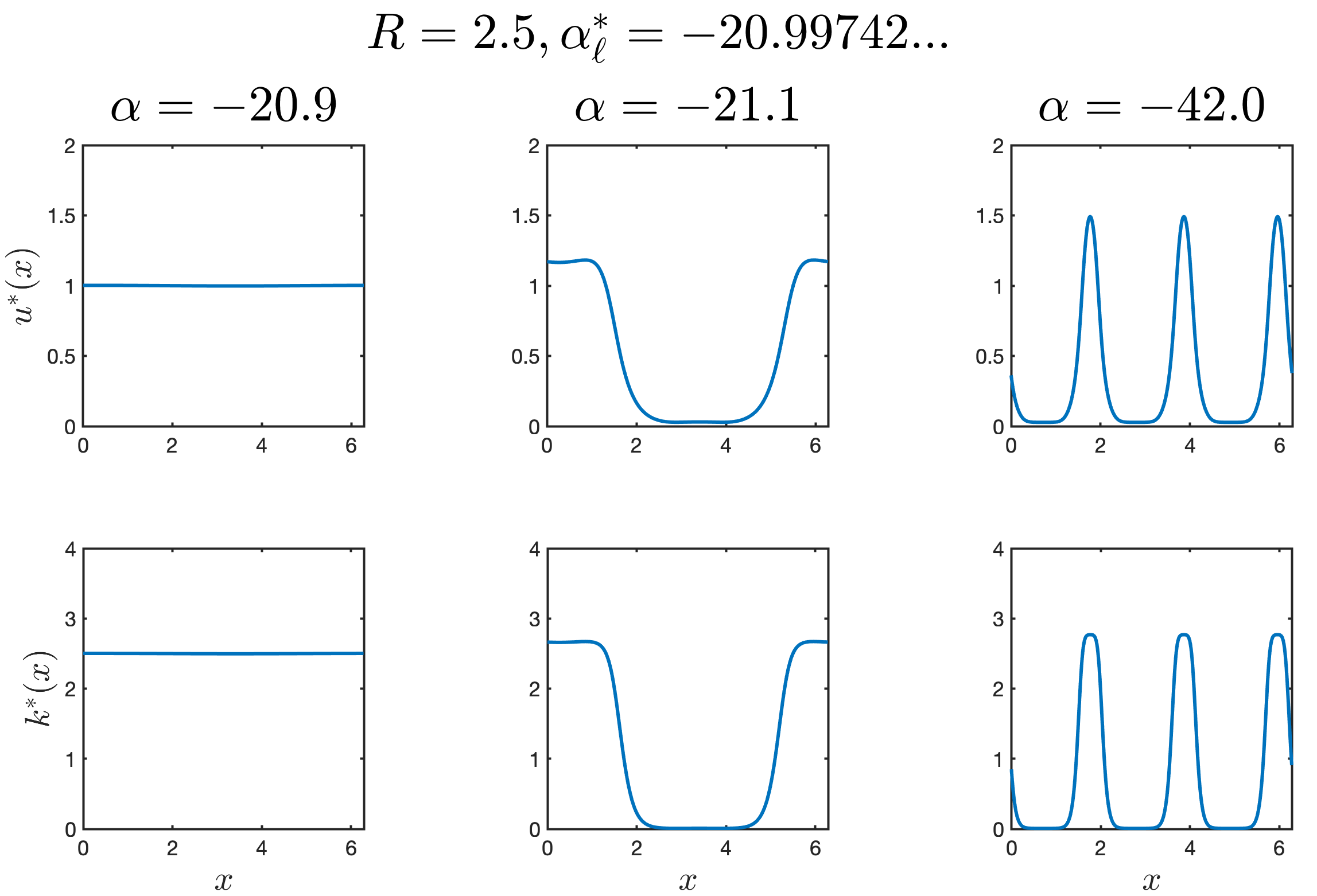}
	\caption{
            The steady state solutions corresponding to Section \ref{sec:case1} for the negative $\alpha$ (aggregation) case. We test $\sim 0.1$ before, $\sim 0.1$ after, and twice the value of the critical value $\alpha_\ell ^*$ given in Figure \ref{fig:1}.
	}
	\label{fig:1b}
\end{figure}

\begin{figure}[ht]
	\centering
	\includegraphics[width=1.0\textwidth]{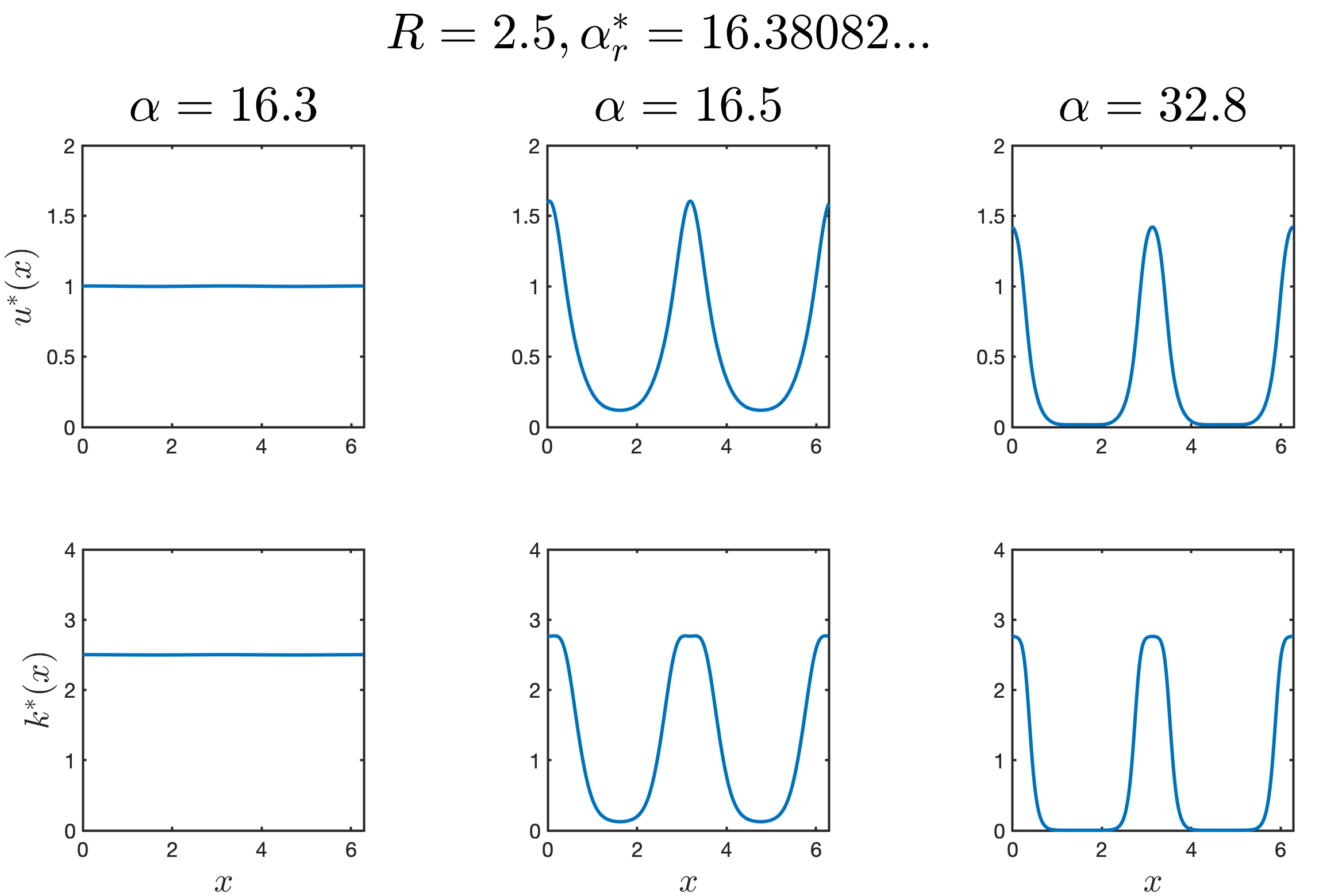}
	\caption{
                    The steady state solutions corresponding to Section \ref{sec:case1} for the positive $\alpha$ (segregation) case. We test $\sim 0.1$ before, $\sim 0.1$ after, and twice the value of the critical value $\alpha_r ^*$ given in Figure \ref{fig:1}.    
	}
	\label{fig:1c}
\end{figure}

\subsection{\texorpdfstring{Case (\romannumeral2) in equation \eqref{1a}}{}}\label{sec:case2}

Let $g(u) = \dfrac{2  \rho u^2}{1 + u}$ in equation \eqref{1a}. Similar to case (\romannumeral1), we have 

\begin{equation}\label{5.10}
\begin{aligned}
    ( u_* &, k_* ) = \left( 1 , \frac{\rho}{\mu + \beta} \right), \ f_{u*} = - 1, \ f_{uu*} = - 2, \ f_{uuu*} = 0, \ g_{u*} = \dfrac{3\rho}{2}, \\ &h_{u*} = \dfrac{3\rho \mu + \rho \beta}{2(\mu + \beta)}, \ h_{k*} = - ( \mu + \beta ),  
    h_{uu*} = g_{uu*} = \dfrac{\rho}{2} ,\ h_{uk*} = - \beta,\\  &h_{kk*} = 0,\ 
    h_{uuu*} = g_{uuu*} = -\dfrac{3\rho}{4},
    h_{uuk*} = h_{ukk*} = h_{kkk*} = 0.
    \end{aligned}
\end{equation}
and
\begin{equation}\label{5.11}
    \alpha_{n}(R) =\frac{Det(J_*)-dh_{k*}n^2}{u_* h_{u*}n^2\frac{\sin (nR)}{2 \pi n R}} = \dfrac{-2(\mu+\beta)^2}{\rho(3 \mu+\beta) \frac{\sin (nR)}{2 \pi n R}}\left(d + \dfrac{1}{n^2}\right).
\end{equation}
In Figure \ref{fig:2}, we again numerically compute $\as{\alpha_l}$ and $\alpha_r$ and plot them with respect to the perceptual radius $R$. In Figures \ref{fig:2b}-\ref{fig:2c}, we again plot the solution profiles just before, just after, and far beyond the analytically calculated bifurcation points $\alpha_l^*$ and $\alpha_r^*$. In this case the stability of the constant steady-state is the same as the result of Theorem \ref{thm:4.1} (which we omit here), while the monotonicity of $\alpha_n(R)$ with respect to parameters is slightly different, we state the theorem below.
\begin{theorem}\label{thm:monotone-2}
Let $\alpha_n(R)$ be defined in \eqref{5.11}. 
\begin{enumerate}
    \item [(\romannumeral1)] Suppose $n \in \Sigma^+$ so that $\dfrac{\sin(nR)}{nR}>0$. Then $\alpha_n(R)<0$, and $\alpha_n(R)$ (in particular, $\alpha_l$) is monotonically increasing with respect to $\rho$, monotonically decreasing with respect to $d$, and is not monotone with with respect to either $\mu$ or $\beta$. 
    \item [(\romannumeral2)] Suppose $n \in \Sigma^-$ so that $\dfrac{\sin(nR)}{nR}<0$. Then $\alpha_n(R)>0$, and $\alpha_n(R)$ (in particular, $\alpha_r$) is monotonically increasing with respect to $d$, monotonically decreasing with respect to $\rho$, and is not monotone with respect to either $\mu$ or $\beta$.
\end{enumerate}
\end{theorem}
This Theorem, similar to Theorem \ref{thm:4.1} for case (i), retains the expected monotonicity properties with respect to $\rho$ and $d$, while we lose monotonicity with respect to $\beta$. Indeed, the curves $\alpha_l$ and $\alpha_r$ are concave with respect to each parameter, suggesting the existence of critical values $\mu^*>0$ and $\beta^*>0$ so that the magnitude of advection required to destabalize the constant steady state is minimal. This highlights a key difference caused by the choice of memory uptake $g(\cdot)$.

\begin{figure}[ht]
	\centering
	\includegraphics[width=0.75\textwidth]{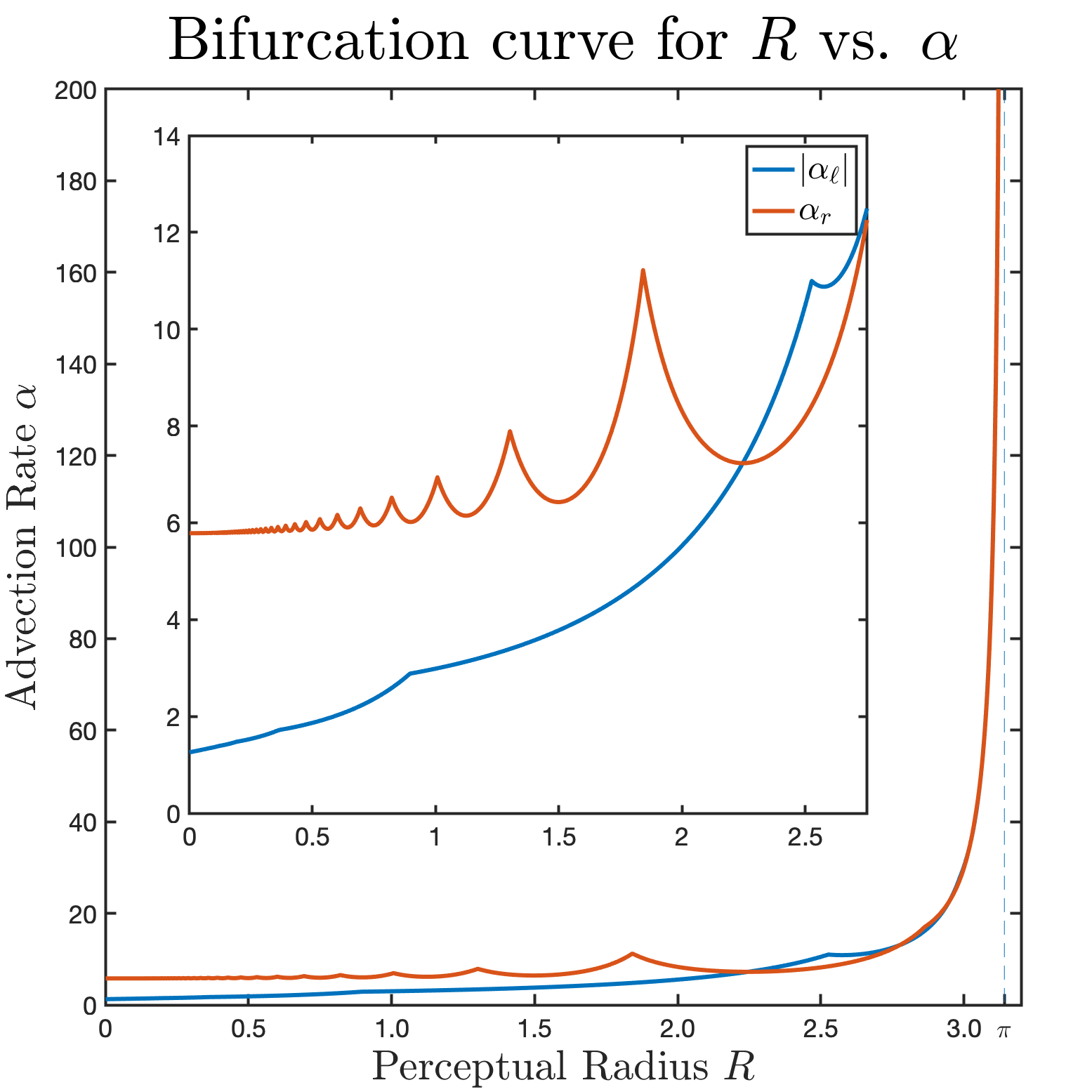}
	\caption{The bifurcation curves for Section \ref{sec:case2} for the perceptual radius $R$ versus the advection rate $\alpha$. In the subsequent Figures \ref{fig:2b}-\ref{fig:2c}, we fix $R=2.5$ and use the bifurcation curves to test values near and far from the critical values $\alpha_r^*$ and $\alpha_\ell ^*$ where the constant steady state is expected to be destabalized. 
	}
	\label{fig:2}
\end{figure}

\begin{figure}[ht]
	\centering
	\includegraphics[width=1.0\textwidth]{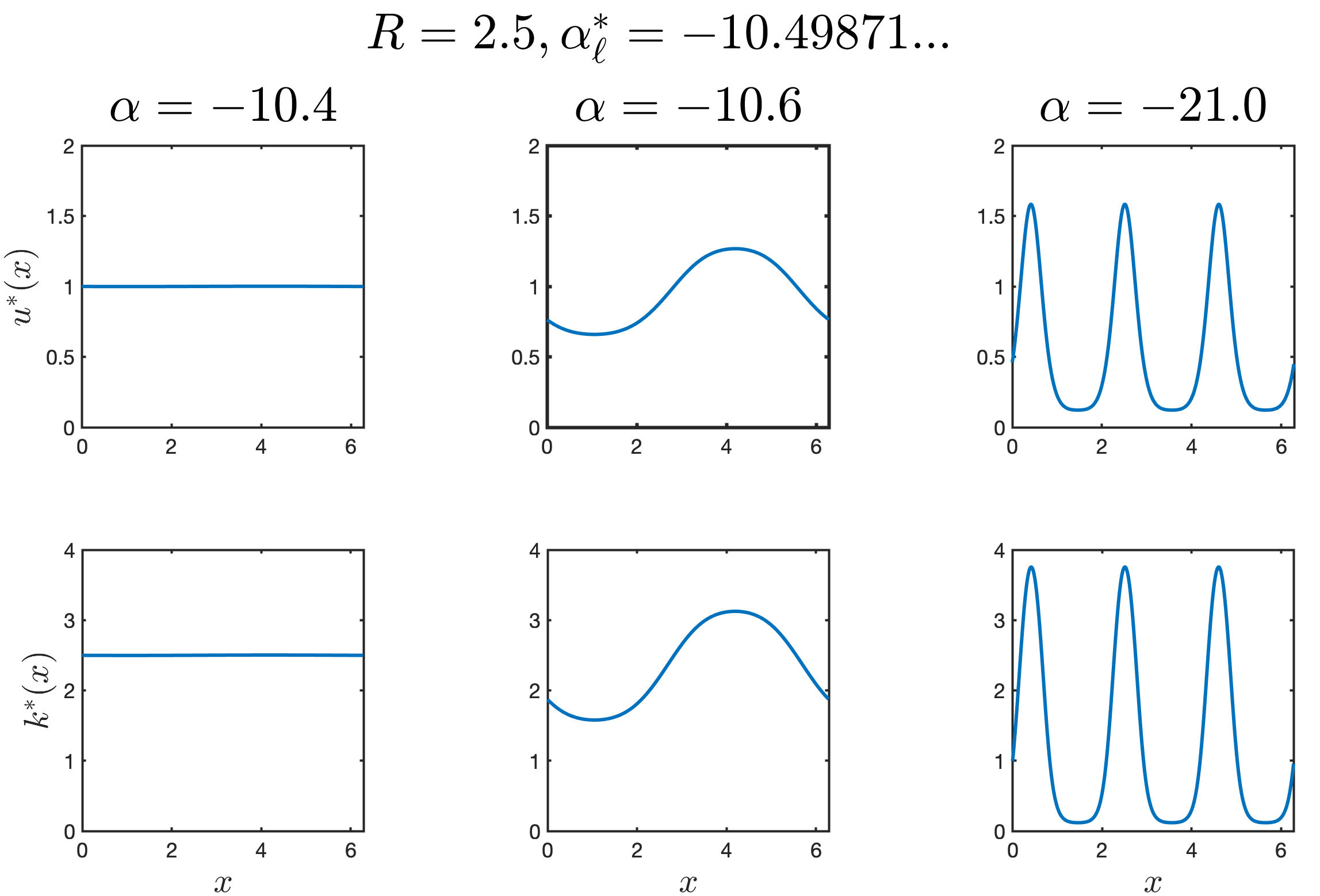}
	\caption{
            The steady state solutions corresponding to Section \ref{sec:case2} for the negative $\alpha$ (aggregation) case. We test $\sim 0.1$ before, $\sim 0.1$ after, and twice the value of the critical value $\alpha_\ell ^*$ given in Figure \ref{fig:2}.
	}
	\label{fig:2b}
\end{figure}

\begin{figure}[ht]
	\centering
	\includegraphics[width=1.0\textwidth]{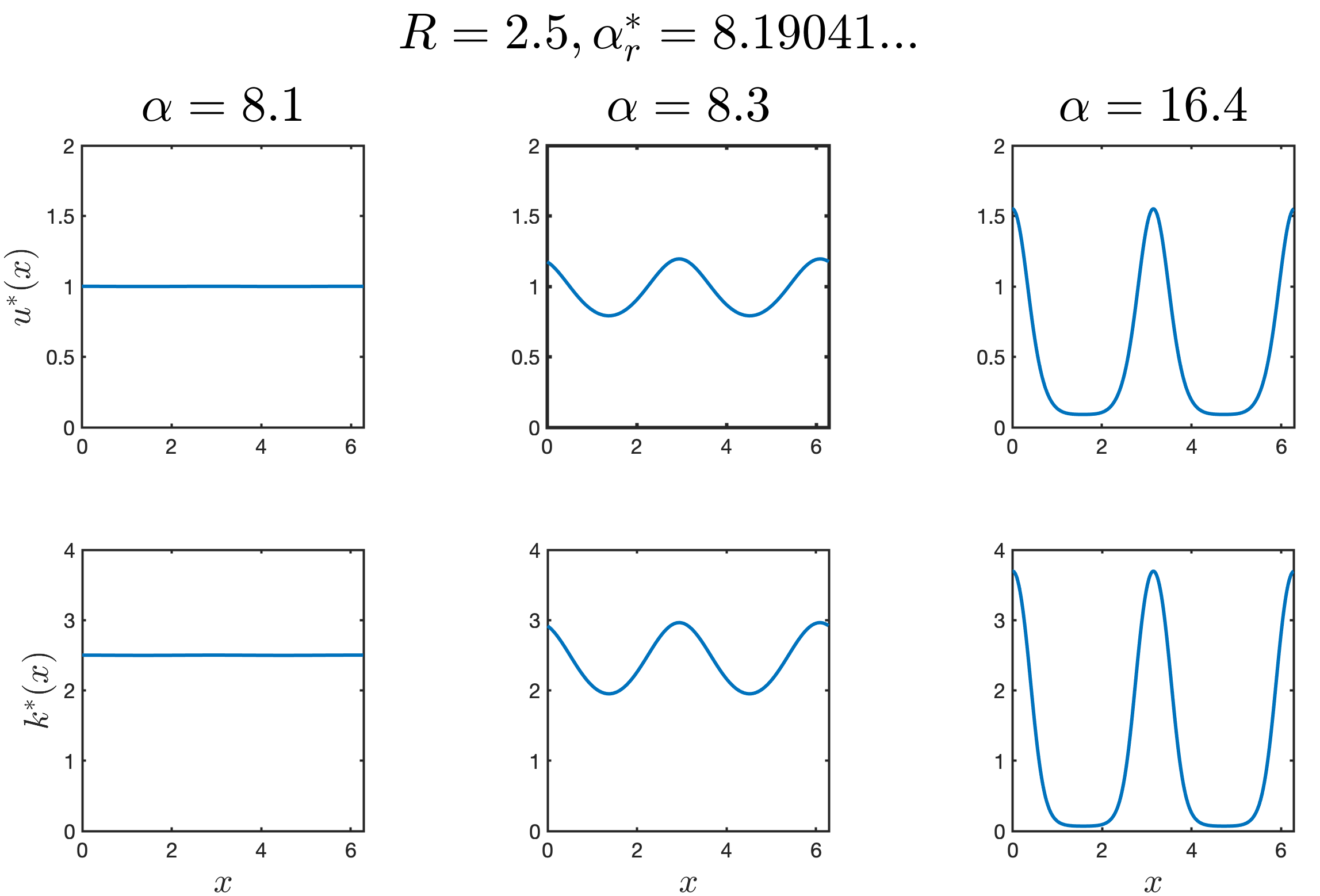}
	\caption{
                    The steady state solutions corresponding to Section \ref{sec:case2} for the positive $\alpha$ (segregation) case. We test $\sim 0.1$ before, $\sim 0.1$ after, and twice the value of the critical value $\alpha_r ^*$.    
	}
	\label{fig:2c}
\end{figure}

\subsection{\texorpdfstring{Case (\romannumeral3) in equation \eqref{1b}}{}}\label{sec:case3}

In this case, we choose $g(u)=\rho u^2$ in equation \eqref{1b}. From \eqref{3.30} and \eqref{3.31}, bifurcation point $\widehat{\alpha}_n(R)$ has the following expression
 \begin{equation}\label{5.12}
 \widehat{\alpha}_{n}(R) =\frac{-(\rho+\mu+\beta)^2}{\kappa\rho (2\mu+\beta) \frac{\sin (nR)}{2 \pi n R}}\left(d-\frac{f'(1)}{l_n}\right).
 \end{equation}
Similar to Theorem \ref{thm:4.1}, we have the following.
\begin{theorem}\label{thm:5.4}
Let $\alpha_n(R), \Sigma^+, \Sigma^-, \alpha_l, \alpha_r$ be defined in \eqref{5.12} and \eqref{3.33}. Then there are non-constant steady-state solutions that bifurcation from the constant solution $( 1 , \frac{\rho}{\rho + \mu + \beta} )$ near $\widehat{\alpha}_{n}(R)$ of system \eqref{1b}. Moreover, the constant solution $( 1 , \frac{\rho}{\rho + \mu + \beta} )$ is locally asymptotically stable when $\widehat{\alpha}_l<\widehat{\alpha}<\widehat{\alpha}_r$ and unstable when $\widehat{\alpha}<\widehat{\alpha}_l$ or $\widehat{\alpha}>\widehat{\alpha}_r$.
\end{theorem}

\begin{theorem}\label{thm:5.5}
Let $\widehat{\alpha}_n(R)$ be defined in \eqref{5.12}. 
\begin{enumerate}
    \item [(\romannumeral1)] Suppose $n \in \Sigma^+$ so that $\dfrac{\sin(nR)}{nR}>0$. Then $\widehat{\alpha}_n(R)<0$, and $\widehat{\alpha}_n(R)$ (in particular, $\widehat{\alpha}_l$) is monotonically increasing with respect to $\kappa$,  monotonically decreasing with respect to $d$, and is not monotone with respect to any of $\rho$, $\mu$ or $\beta$;
    \item [(\romannumeral2)] Suppose $n\in \Sigma^-$ so that $\dfrac{\sin(nR)}{nR}<0$. Then $\widehat{\alpha}_n(R)>0$, and $\widehat{\alpha}_n(R)$ (in particular, $\widehat{\alpha}_r$) is monotonically increasing with respect to $d$, monotonically decreasing with respect to $\kappa$, and is not monotone with respect to any of $\rho$, $\mu$ or $\beta$.
\end{enumerate}
\end{theorem}

We again compare to Theorem's \ref{thm:4.1}-\ref{thm:monotone-2}: the monotonicity with respect to $d$ remains, while a quadratic growth for the memory uptake function causes all other previously monotone cases to be nonmonotone! In case (iii), we also have a new parameter $\kappa$, which is the theoretical maximal memory capacity of the organism. It is biologically reasonable, therefore, for an increase in this memory capacity to decrease the magnitude of advection required to destabalize the constant steady state.

In Figure \ref{fig:3}, we plot the bifurcation curves $\alpha_l(R)$ and $\alpha_r(R)$. Figures \ref{fig:3b}-\ref{fig:3c} again show the solution profiles as done in cases (\romannumeral1) and (\romannumeral2). 

\begin{figure}[ht]
	\centering
	\includegraphics[width=0.75\textwidth]{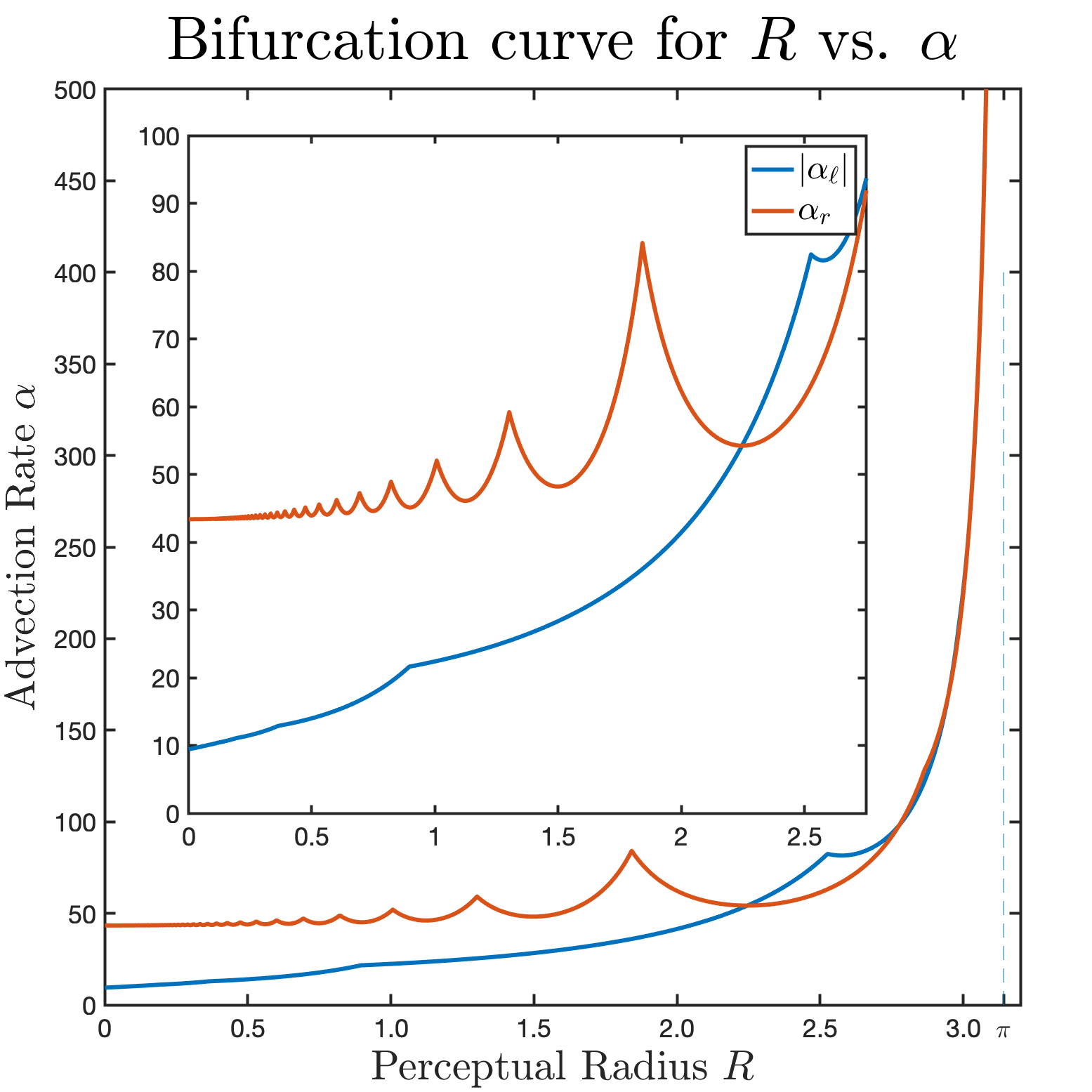}
	\caption{The bifurcation curves for Section \ref{sec:case2} for the perceptual radius $R$ versus the advection rate $\alpha$. In the subsequent Figures \ref{fig:3b}-\ref{fig:3c}, we fix $R=2.5$ and use the bifurcation curves to test values near and far from the critical values $\alpha_r^*$ and $\alpha_\ell ^*$ where the constant steady state is expected to be destabalized. 
	}
	\label{fig:3}
\end{figure}

\begin{figure}[ht]
	\centering
	\includegraphics[width=1.0\textwidth]{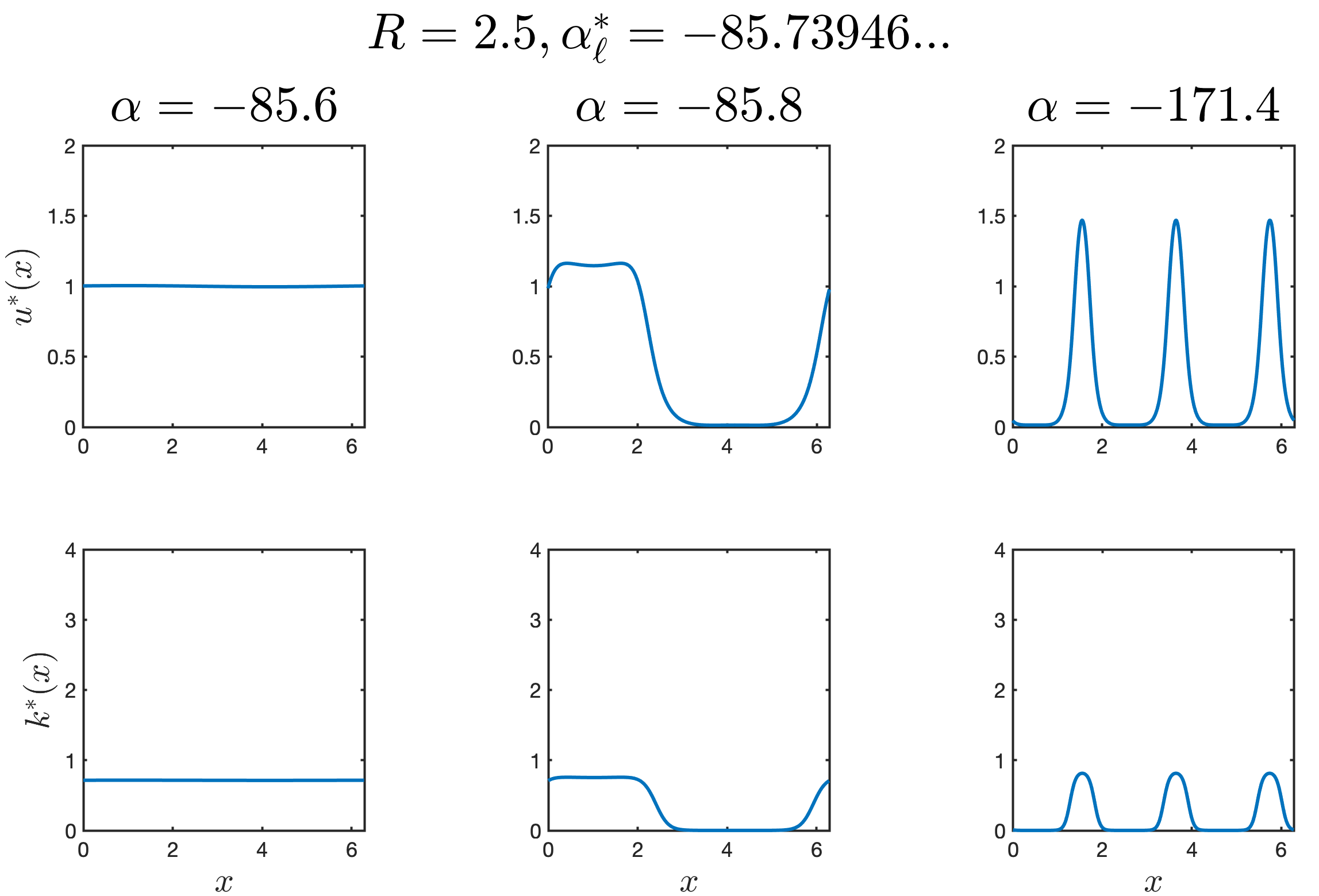}
	\caption{
            The steady state solutions corresponding to Section \ref{sec:case3} for the negative $\alpha$ (aggregation) case. We test $\sim 0.1$ before, $\sim 0.1$ after, and twice the value of the critical value $\alpha_\ell ^*$.
	}
	\label{fig:3b}
\end{figure}

\begin{figure}[ht]
	\centering
	\includegraphics[width=1.0\textwidth]{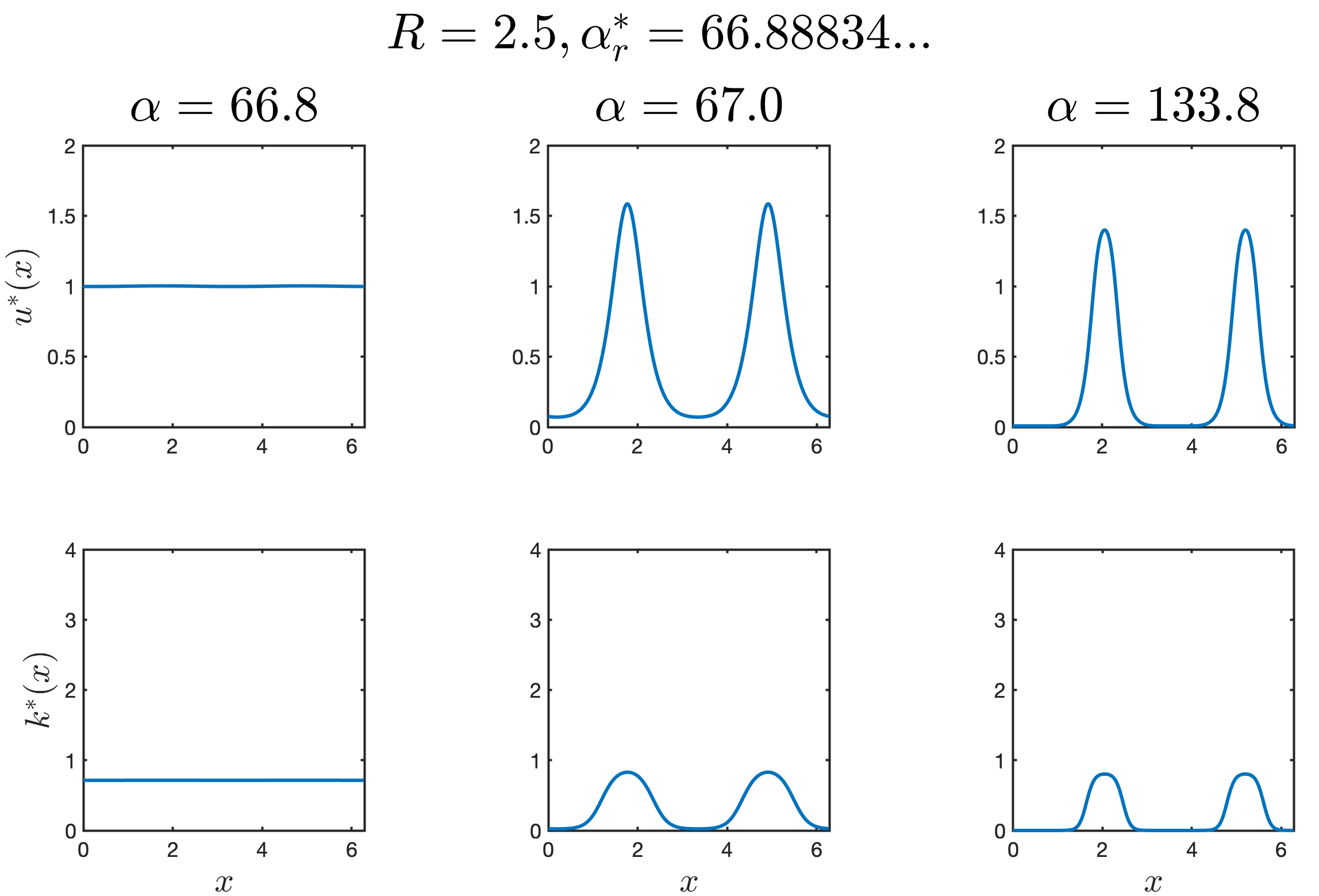}
	\caption{
                    The steady state solutions corresponding to Section \ref{sec:case3} for the positive $\alpha$ (segregation) case. We test $\sim 0.1$ before, $\sim 0.1$ after, and twice the value of the critical value $\alpha_r ^*$.    
	}
	\label{fig:3c}
\end{figure}

\section{Discussion}\label{sec:discussion}

The role of spatial memory in driving the movement of animals has long been of interest to both empirical ecologists \cite{faganetal2013} and mathematical modellers \cite{wangsalmaniw2022}. In this work, we consider the incorporation of a nonlocal advection term in the PDE to model movement in response to remembered space use, where the memory map is described dynamically by an additional ODE. The nonlocal advection term is crucial from both biological and mathematical standpoints \cite{painter2023biological}. Biologically, it more accurately captures an essence of how organisms sense their surrounding environment and make movement decisions based on that information \cite{martinezgarciaetal2020}. This is a useful step forward in our mathematical representation of animal movement ecology, making the modelling formulation more applicable to what is observed in the natural world. Mathematically, however, nonlocality introduces technical difficulties which deserve a careful and robust study. 

One of the significant contributions of this paper is the establishment of a well-posedness result, proving the existence and uniqueness of a global solution, ruling out the possibility of a finite-time blow up. In particular, we showed the existence and uniqueness of a solution even when considering the discontinuous top-hat detection function. This result broadens the existing literature in a crucial way, providing answers to some open questions found in \cite{wangsalmaniw2022}. Prior to this work, the existence of solutions for this specific class of models remained an open question. Our result not only bridges this gap but also opens doors for more complex models incorporating different types of detection functions.

Another significant contribution lies in our robust bifurcation and spectral analyses. Previous efforts, motivated more directly by the ecological application, have often relied solely on a linear stability analyses \cite{briscoeetal2002, potts2016territorial}, which can be insufficient in scenarios where the essential spectrum is nonempty. This does not occur in classical reaction-diffusion systems, but may be the case when nonlocal advection is introduced. Moreover, the comprehensive approach used here reveals a more nuanced understanding of the system's stability, describing more qualitative features of the solution profile near these critical bifurcation points.

From this bifurcation analysis, we establish a number of monotonicity and non-monotonicity results for the critical bifurcation parameters, with the monotonic properties depending on the functional form given chosen for the memory uptake rate $g(\cdot)$. In the special cases considered here, a bounded functional form has the most monotonicity properties, while a roughly linear or quadratic functional form appears to remove most of the monotonicity properties. This suggests the existence of critical values so that the magnitude of advection required to destabalize is minimized, an interesting feature that deserves future study. For example, is the critical value dependent on whether the population aggregates or segregates? 

Our numerical simulations using a pseudo-spectral method further complement our analytical results. These numerical methods are particularly well-suited for dealing with nonlocal advection-diffusion problem, demonstrating some of the interesting dynamics found in these models.

While this work provides new insights and fills existing gaps in the literature, further studies are needed to explore more general functional forms, higher-dimensional space, and for domains with a physical boundary.  From a biological perspective, there are various aspects to memory at play that are not modelled here \cite{faganetal2013}.  In reality, animals' advective tendencies will not simply be towards (or away from) areas they have previously visited.  Rather they will assess the quality of those places -- e.g. whether they contain access to food or shelter, if they have had aggressive or favourable encounters there -- and adjust their advective tendencies accordingly \cite{lewisetal2021}.  Our work paves the way for analysing these more detailed and realistic memory effects.  

One tricky, yet important, feature will be the inclusion of heterogeneous landscapes, for example where some areas are better than others for foraging or hiding from predators \cite{vanmoorteretal2009, merkleetal2014}.  Analysis of nonlinear PDEs often takes place on a homogeneous environment, but to connect better to the ecological community, theory on pattern formation in heterogeneous environments will be of fundamental importance \cite{krause2020one}. Additionally, it will be important to account for between-population interactions via multi-species models with explicit inclusion of memory processes \cite{pottslewis2019}.  A possible way in to this would be to analyse existing models on territory, some which are simply multi-species extensions of the model analysed here \cite{pottslewis2016,potts2016territorial}, by placing these on solid mathematical foundations through developing existence theory, and gaining greater insights into territorial pattern formation through rigorous spectral and bifurcation analyses.

\section{Acknowledgments}\label{sec:ack}

The authors thank Mark Lewis for his insightful comments on model development. DL was partially supported by a HIT PhD Scholarship and the University of Alberta. YS is supported by NSERC grant PDF-578181-2023. JRP acknowledges support of Engineering and Physical Sciences Research Council (EPSRC) grant EP/V002988/1. JS is partially supported by the U.S. National Science Foundation grants OCE-2207343 and DMS-185359. HW gratefully acknowledges support from the Natural Sciences and Engineering Research Council of Canada (Discovery Grant RGPIN-2020-03911 and NSERC Accelerator Grant RGPAS-2020-00090) and the Canada Research Chairs program (Tier 1 Canada Research Chair Award).

%%%%%%%%%%%%%%%%%%%%%%%%%%%%%%%%%%%%%%%%%%%%%%%%%%%%%%%%%%%%%%%%%%%%%%%%%%%%%%%%%%%%%%%%%%%%%%%%%%%%%%%%%%%%%%%%%%%%%%%%%%%%%%%%%%%%%%%%%%%%%%%%%%%%%%%%%%%%%%%%%%%%%%%%%%%%%%%%%%%%%%%%%%%%%%%%%%%%%%%%%%%%%%%%%%%%%%%%%%%%%%%%%%%%%%%%%%%%%%%%%%%%%%%%%%%%%%%%%%%%%%%%%%%%%%%%%%%%%%%%%%%%%%%%%%%%%%%%%%%%%%%%%%%%%%%%%%%%%%%%%%%%%%%%%%%%%%%%%%%%%%%%%%%%%%%%%%%%%%%%%%%%%%%%%%%%%%%%%%%%%%%%%%%%%%%%%%%%%%%%%%%%%%%%%%%%%%%%%%%%%%%%%%

\bibliographystyle{unsrtnat}
\bibliography{references}  %%% Uncomment this line and comment out the ``thebibliography'' section below to use the external .bib file (using bibtex) .

\end{document}